\documentclass[10 pt, reqno]{amsart}
\usepackage{amsaddr}
\usepackage{dsfont}
\usepackage{enumerate}
\usepackage{amsfonts}
\usepackage{amsthm}
\usepackage{amssymb}
\usepackage{latexsym}
\usepackage{amsmath}
\usepackage{mathrsfs}
\usepackage{color}
\usepackage{afterpage}
\usepackage{a4wide}
\usepackage[utf8]{inputenc}
\usepackage{hyperref}

\theoremstyle{plain}
\newtheorem{tw}{Theorem}[section]

\newtheorem {lem} [tw]{Lemma}
\newtheorem {prop}[tw] {Proposition}

\newtheorem{cor}[tw]{Corollary}

\theoremstyle{definition}

\newtheorem {rem} [tw]{Remark}

\newcommand{\cst}{\ifmmode\mathrm{C}^*\else{$\mathrm{C}^*$}\fi}

\newcommand{\M}{\mathsf{M}}
\newcommand{\h}{\mathsf{H}}
\newcommand{\hu}{\mathsf{H}_U}

\newcommand{\hr}{\mathsf{H}_{\mathbb{R}}}
\newcommand{\hri}{\mathsf{H}_{\mathbb{R}}^{(i)}}
\newcommand{\hrj}{\mathsf{H}_{\mathbb{R}}^{(j)}}

\newcommand{\hc}{\mathsf{H}_\mathbb{C}}
\newcommand{\kr}{\mathsf{K}_{\mathbb{R}}}
\newcommand{\lr}{\mathsf{L}_{\mathbb{R}}}

\newcommand{\fth}{\mathcal{F}_T(\mathsf{H})}
\newcommand{\vonT}{\Gamma_T(\mathsf{H}_\mathbb{R}, U_t)}

 \newcommand{\N}{\mathbb{N}}
\newcommand{\R}{\mathbb{R}}
\newcommand{\C}{\mathbb{C}}

\newcommand{\bc} {\mathbb{C}}

\newcommand{\br}{\mathbb{ R}}

\DeclareMathOperator{\op}{op}
\DeclareMathOperator{\dom}{Dom}
\DeclareMathOperator{\Span}{span}

\newcommand {\id} {{\textrm{id}}}
\newcommand{\inv}{{\text{inv}}}

\newcommand{\Hil}{\mathsf{H}}

\newcommand{\oned}{\{1,\ldots,d\}}

\newcommand{\FockTH}{\mathcal{F}_T(\Hil_U)}
\newcommand{\FockTHalg}{\FockTH_{\textup{alg}}}

\newcommand{\mlg}{\mathsf{M}}

\numberwithin{equation}{section}

\theoremstyle{plain}

\begin{document}


\author[Manish Kumar]{Manish Kumar}  

\email{mkumar@impan.pl}

\title[Mixed $q$-Araki-Woods von Neumann algebras]{Conjugate variables approach to {\em mixed} $q$-Araki-Woods Algebras: Factoriality and non-injectivity}
\address{Institute of Mathematics of the Polish Academy of Sciences,
 Warsaw, Poland}

\begin{abstract}
We establish   factoriality and non-injectivity in full generality for the mixed $q$-Araki-Woods von Neumann algebra associated to a  separable real Hilbert space $\mathsf{H}_{\mathbb{R}}$ with $\dim\mathsf{H}_{\mathbb{R}}\geq 2$, a strongly continuous one parameter group of orthogonal transformations on $\mathsf{H}_\mathbb{R}$,  a direct sum decomposition $\mathsf{H}_{\mathbb{R}}=\oplus_{i}\mathsf{H}_{\mathbb{R}}^{(i)}$, and  a real symmetric matrix $(q_{ij})$ with $\sup_{i,j}|q_{ij}|<1$. This is achieved 
by first proving the existence of conjugate variables for a finite number of generators of the algebras (when $\dim\mathsf{H}_{\mathbb{R}}<\infty$), following the lines of Miyagawa-Speicher and Kumar-Skalski-Wasilewski. The conjugate variables belong to the  factors in question and satisfy certain Lipschitz conditions.
\end{abstract}

\subjclass[2020]{Primary: 46L36; Secondary  46L10, 46L53, 46L65}

\keywords{mixed $q$-Araki-Woods von Neumann algebra, factoriality, non-injectivity, conjugate variables}
\maketitle

\section{Introduction}
We study the factoriality and non-injectivity question of mixed $q$-Araki-Woods von Neumann algebras. These are  non-tracial counterpart of mixed $q$-Gaussian algebras introduced by Speicher \cite{Spe} and Bo\.zejko-Speicher \cite{BS} which followed  their definition of (non-mixed) $q$-Gaussian algebras in \cite{BS91}. Mixed $q$-Gaussian algebras are associated with  two data $(\hr, (q_{ij}))$  where $\hr$ is a real separable Hilbert space and $-1<q_{ij}=q_{ji}<1$ are real numbers for $1\leq i,j\leq \dim\hr$ such that $\sup_{i,j}|q_{ij}|<1$. The von Neumann algebra is then generated   by a family of operators $\{l_i+l_i^*\}_{1\leq i\leq \dim\hr}$ acting
 on certain deformed Fock space where $l_i$'s obey  the mixed $q_{ij}$ commutation relations:
\begin{align*}
    l_il_j^*-q_{ij}l_j^*l_i=\delta_{ij}.
\end{align*}
 The non-mixed scenario corresponds to the case where $q_{ij}=q$ for all $i,j$. The structure of these algebras have attracted  deep analysis over the decades; in  particular  factoriality of (mixed) $q$-Gaussian algebras is  work of many hands. Ricard \cite{Eric} established the factoriality of $q$-Gaussian algebras in full generality (for $\dim\hr\geq 2$) while several partial results were obtained earlier in \cite{BKS,  Ilona, Sniady}. Following the ideas of \cite{Eric}, Skalski and Wang in \cite{SW} solved the factoriality question of mixed $q$-Gaussian algebras in full generality ($\dim\hr\geq 2$) while again some partial results preceded this in \cite{Kro, NZ}. The citations listed here are by no means complete for the study of various other structural properties of the Gaussian algebras.  

Both $q$-Gaussian and mixed $q$-Gaussian von Neumann algebras are tracial and are natural deformations of Voiculescu's free Gaussian functor associated to  real Hilbert spaces. Another natural generalization of the free Gaussian functor was introduced by Shlyakhtenko \cite{Sh1} in the non-tracial framework  which can also be regarded as  analogues of injective factors coming from CAR relations. The associated algebras depend on one parameter groups of orthogonal transformations $U_t$ on  real Hilbert spaces $\hr$.  They are non-injective factors of type $\mathrm{III}$ as soon as $\dim\hr\geq 2$ and $U_t$ is non-trivial, and they are called {free Araki-Woods factors}.

Hiai in  \cite{Hiai} combined the functor of \cite{BS91} and \cite{Sh1} to produce the so called $q$-Araki-Woods von Neumann algebras associated to  $(\hr, U_t, q)$. These algebras have equally received a lot of attention, particularly because they provide a concrete approach to the study of certain type $\mathrm{ III}$ factors.  On one hand factoriality and non-injectivity  was determined completely for the free Araki-Woods algebras (i.e. when $q=0$) in \cite{Sh1}, whereas the same for $q$-deformation counterparts for $q\neq 0$ appears to be more involved, with only  partial results  obtained   over the years  in  \cite{Hiai,Brent0, BM, SW, BMRW, KSW}.  Remark that the factoriality was  established by Bikram-Mukherjee-Ricard-Wang in \cite{BMRW} for $\dim\hr\geq 3$, and for $\dim\hr\geq 2$ with $U_t$ ergodic and  $q$ large.  The complete answer to the  factoriality and non-injectivity question of such algebras  have been settled   recently by   Skalski,  Wasilewski and the author in \cite{KSW}. Unlike the free case scenario, various  structural properties of $q$-deformed algebras like fullness, absence of Cartan subalgebras, solidity etc are still far from full understanding.

In \cite{BKM}, Bikram, Kumar R. and Mukherjee further extended Hiai's construction to now include the \textit{mixed} $q$-Gaussian functor of \cite{BS} as well as the functor of \cite{Sh1}.  The concerned von Neumann algebras are called {\em mixed $q$-Araki-Woods von Neumann algebras}. Now the algebras are associated with data consisting of $(U_t, \hr^{(i)}, (q_{ij})_{i,j\in {N}}, T)$ where  ${N}$ is a countable set, $\hr^{(i)}, i\in {N}$ is a real {separable} Hilbert space, parameters $q_{ij}\in (-1,1)$  {satisfy} $\sup_{i,j\in N}|q_{ij}|<1$, $U_t$ is a strongly continuous one parameter group of orthogonal operators on $\hr:=\oplus_{i\in {N}}\hr^{(i)}$ such that each $\hri$ is invariant under $U_t$, and $T\in \mathbf{B}(\hr\otimes\hr)$ is a Yang-Baxter operator which interpolates $\hr^{(i)}\otimes\hr^{(j)}$ and $\hr^{(j)}\otimes\hr^{(i)}$ with the twist $q_{ij}$. The mixed $q$-Araki-Woods algebra $\Gamma_T(\hr, U_t)$ is then generated by a set of operators $\{a(\xi)+a^*(\xi); \xi\in \hr^{(i)} , i\in {N}\}$ acting on a deformed Fock-space, satisfying the following commutation relations:
\begin{equation*}
    a(\xi) a^*(\eta)-q_{ij}a^*(\eta)a(\xi)=\langle\xi,\eta\rangle_U
\end{equation*}
for $\xi\in \hr^{(i)}$ and $\eta\in \hr^{(j)}$ where $\langle\cdot,\cdot\rangle_U$ is a twisted scalar product on the complexification of $\hr$ with the twist arising out of the representation $U_t$ (see the formal construction in Section \ref{sec:preliminaries}).

 Previous investigations of mixed $q$-Araki-Woods algebras in \cite{BKM} established  factoriality and the type classification of $\Gamma_T(\hr, U_t)$ under different conditions i.e. whenever $U_t$ is weakly mixing or $U_t$ admits a fixed non-zero vector or when $U_t$ is almost-periodic on an infinite dimensional space $\hr$ with  the set of  eigenvalues of $A$ having a non-zero limit point where $A$ is the analytic generator of $U_t$. They also prove non-injectivity of the algebras in certain cases i.e.  when the weakly mixing part of $U_t$ is non-zero. Several other nice properties  are illustrated in \cite{BKM} mostly generalizing those  of their sister algebras of non-mixed situation.
 
 However the case when the dimension of the underlying  Hilbert space $\hr$ is finite could not be fully understood. In this paper, we finally settle  the factoriality and non-injectivity question of mixed $q$-Araki-Woods algebras in full generality.   
We summarize our main results in the following theorem which is a combination of Theorem \ref{thm:factor} and Theorem \ref{thm:noninj}:
\begin{tw}
    Let $(\hr=\oplus_{i\in{N}}\hri, U_t, (q_{ij})_{i,j\in N}, T)$ be given with $\dim\hr\geq2$, $q=\max_{i,j}|q_{ij}|<1$ and $N$ countable. Then the von Neumann algebra $\Gamma_T(\hr, U_t)$ is a non-injective factor of type
    \begin{align*}
\left\{\begin{array}{ll} \mathrm{III}_1 & \text{ if } G=\mathbb{R}^{\times}_{\ast}, \\ \mathrm{III}_{\lambda} & \text{ if } G= \lambda^{\mathbb{Z}}, 0<\lambda<1, \\ \mathrm{II}_1 & \text{ if } G=\{1\},  \end{array}   \right.	  
    \end{align*}
    where $G< \mathbb{R}^{\times}_{\ast}$ is the closed subgroup generated by the eigenvalues of the generator $A$ of  $(U_t)_{t \in \R}$.
\end{tw}
Our approach towards the proof utilises the concepts  of conjugate variables following the lines of arguments by Miyagawa-Speicher \cite{MS}, Kumar-Skalski-Wasilewski \cite{KSW} and Nelson \cite{Brent}. The key idea of  previous attempts {towards} factoriality in \cite{BKM} is  the study of mixing properties of certain subalgebras which allows one to control the relative commutant of the centralizer, whereas  our methods take completely different route. 

We begin with the construction and basic properties of mixed $q$-Araki-Woods von Neumann algebras in Section \ref{sec:preliminaries}. We  recollect the notions of  dual/conjugate variables in Section  \ref{sec:dual variables} and then exhibit in Theorem \ref{thm:existence of conjugate variables} the existence of  a  system of conjugate variables for a set of generators for the von Neumann algebra $\Gamma_T(\hr, U_t)$ whenever  $\dim\hr<\infty$.  The conjugate variables are shown to lie in the ambient von Neumann algebras (Corollary \ref{cor:conjugates variables are in M Omega}). 

Moreover we prove in  Corollary \ref{cor: Lipstisz conjugate} that the conjugate variables are  Lipschitz  (see Section \ref{sec:dual variables} for the definition of Lipschitz condition).
Remark that the Lipschitz condition was not addressed in \cite{KSW} even for $q_{ij}=q$ case. In this paper we do not precisely indicate an application of the Lipschitz conjugate variables, nevertheless we believe that the Lipschitz condition may provide  pleasant applications similar to the tracial case and we are working to achieve some in our upcoming work. To mention one striking application from the tracial settings, in conjugation with  abstract results of \cite{DI}, the Lipschitz condition allowed \cite{MS} to furnish a different proof for  the absence of Cartan subalgebras in $q$-Gaussian factors with finite symbols.

Finally in Section \ref{sec:main}, the existence of conjugate variables allows us to invoke the abstract results of Nelson \cite{Brent} to prove factoriality and non-injectivity of $\Gamma_T(\hr, U_t)$ for $\hr$ finite dimensional.   For a general (separable) Hilbert space $\hr$ with $U_t$ being almost periodic (remaining cases are  dealt in \cite{BKM}),  an inductive limit argument is used to deduce factoriality while some additional arguments are added to deduce non-injectivity. 

After the completion of this article, the author learnt that parts of the results about non-injectivity of $\Gamma_T(\hr, U_t)$ have been obtained independently by  \cite{BKM2}. I thank Rahul Kumar for making me aware of their results from \cite{BKM2} and for sending the draft which was not available online. More precisely  they prove non-injectivity of $\Gamma_T(\hr, U_t)$ under certain  conditions on the dimension of a fixed spectral projection of the analytic generator $A$ of $U_t$ (see \cite[Theorem 4.3, Corollary 4.4]{BKM2}). However their investigations  which still leave a large number of cases open are inspired from the ideas of \cite{Nou1}  proving the equivalent conditions of non-semi-discreetness,   and they do not take into account the point of view of conjugate and dual variables as we have done.

\section{Construction of mixed $q$-Araki-Woods von Neumann algebras}\label{sec:preliminaries}
We briefly describe the construction of mixed $q$-Araki-Woods von Neumann algebras and refer the readers to \cite{BKM} for a detailed treatment. As a convention, all Hilbert spaces are separable and inner products are linear in the second variable. For any real or complex Hilbert space $H$, $\mathbf{B}(H)$ denotes the algebra of bounded operators on $H$.

 Let $\hr$ be a real Hilbert space, and let $U_t$ be a  strongly continuous orthogonal representations of $\R$ on $\hr$. Consider the complexification $ \hc= \hr\otimes_\R\C$ of $\hr$, whose inner product is denoted by $\langle\cdot,\cdot\rangle_{\hc}$.  
 Identify $\hr$ inside $ \hc$ by $\hr\otimes 1$, so that $  \hc=\hr+i\hr$.
Given $\xi=\xi_1+i\xi_2\in  \hc$ for some $\xi_1,\xi_2\in\hr$, write the canonical involution $\bar{\xi}$ as
\begin{equation*}
    \bar{\xi}=\xi_1-i\xi_2.
\end{equation*}
We  extend $U_t$ to a strongly continuous one parameter group of unitaries on the  complexification $ \hc$, which is still denoted by $U_t$. Let $A$ denote the positive, self-adjoint, and non-singular analytic (unbounded) operator on $\hc$ obtained via Stone's theorem, which satisfies $U_t=A^{it}$ for all $t\in\R$. Note that $\overline{A\xi}=A^{-1}\bar{\xi}$ for any $\xi$ in the domain of $A$. In particular, $\sigma(A)\cap (0,\infty)$ is symmetric around $1$ i.e. for $\lambda\neq 0$, $\lambda\in\sigma(A)$ if and only if $1/\lambda\in\sigma(A)$, where $\sigma(A)$ is the spectrum of $A$.

Following Shlyakhtenko \cite{Sh1},  consider the deformed inner product $\langle\cdot,\cdot\rangle_U$ on $\hc$ defined by  
\begin{equation*}
    \langle\xi,\zeta\rangle_U=\langle \xi, 2(1+A^{-1})^{-1}\zeta\rangle\quad\forall\xi,\zeta\in \hc.
\end{equation*}
We denote the completion of $\hc$ with respect to $\langle\cdot,\cdot\rangle_U$  by $\hu$. The space $(\hr,\|\cdot\|_{ \hr})$ is naturally identified as a subspace of $\hu$ such that $\hr+i\hr$ is dense in $\hu$ and $\hr\cap i\hr=\{0\}$.
Let ${N}$ denote a countable set $\{1,\ldots,r\}$ for $r\in\N$ or $r=\infty$. Fix a decomposition  of $\hr$ 
\begin{equation*}
    \hr=\bigoplus_{i\in{N}}\hri
\end{equation*}
into a direct sum of invariant subspaces of $U_t.$ Fix $-1<q_{ij}=q_{ji}<1$ for $i,j\in{N}$ such that $\sup_{i,j\in{N}}|q_{ij}|<1.$ We will often write $q(i,j)$ instead of $q_{ij}$. For each $i,j\in N$ define the operator $T_{ij}:\hri\otimes\hrj\to\hrj\otimes\hri$ to be the bounded linear extension of the assignment \begin{equation}\label{eq:expression for T}\xi\otimes\eta\mapsto q_{ij}\eta\otimes\xi\end{equation}
for $\xi\in\hri,\eta\in\hrj$. Set $T=\oplus_{i,j\in{N}}T_{ij}\in \mathbf{B}(\hr\otimes\hr)$. Naturally $T$ extends linearly to $\hc\otimes\hc$ and then to a bounded operator on $\hu\otimes\hu$ in such a way that $T(\hu^{(i)}\otimes\hu^{(j)})\subseteq\hu^{(j)}\otimes\hu^{(i)}$. Moreover $T$ satisfies the following properties:
\begin{align*}
    &T=T^*,\\
    &\|T\|<1, \mbox{ and }\\
    &(I\otimes T)( T\otimes I)( I\otimes T)= (T\otimes I)(I\otimes T)(T\otimes I)
\end{align*}
where $I$ denotes the identity on $\hu$. The last of the above conditions is called Yang-Baxter equation. For each $k\in\N$, define $T_k\in \mathbf{B}(\hu^{\otimes {k+1}})$ by
\begin{equation}\label{eq: notation for T_k}
    T_k={\underbrace {I\otimes\cdots\otimes I}_{k-1}}\otimes T
\end{equation}
and by amplification also on $\hu^{\otimes n}$ for all $n\geq k+1$ (which again is denoted by $T_k$ by abuse of notation).
For $n\in \N$, let $S_n$ denote the symmetric group of $n$ symbols and let $\tau_k$ be the transposition between $k$ and $k+1$ for $1\leq k\leq n-1.$ Let $\Phi:S_n\to \mathbf{B}(\hu^{\otimes n})$ be the unique (quasi-multiplicative) map satisfying $\Phi(1)=I$ 
 and $\Phi(\tau_k)=T_k$, and define $P^{(n)}\in \mathbf{B}(\hu^{\otimes n})$ by 
 \begin{equation*}
     P^{(n)}=\sum_{\sigma\in S_n}\Phi(\sigma).
 \end{equation*}
Then $P^{(n)}$ is a strictly positive operator  on $\hu^{\otimes^n}$ (\cite{BS}); Consider the inner product  on $\hu^{\otimes n}$ defined by 
\begin{equation*}
    \langle \xi,\eta\rangle_T=\langle \xi, P^{(n)}\eta\rangle_{\hu^{\otimes n}}
\end{equation*}
and denote by $\hu^{\otimes_T^n}$ the completion of $\hu^{\otimes n}$ with respect to $\langle\cdot,\cdot\rangle_T$. Define the $T$-deformed Fock space $\mathcal{F}_T(\h_U)$ to be  the direct sum $\oplus_{n\geq 0}\hu^{\otimes_T^n}$, where $\hu^0$ is understood to be the one-dimensional space $\C\Omega$ for a distinguished unit vector $\Omega$. 
We now consider the free left   annihilation operator $l(\xi)$ for any $\xi\in \hu$ given on a dense subspace of $\FockTH$   by the assignment
\begin{align}\label{eq:free annihilation operator}
    l(\xi)\Omega=0, \;\;l(\xi)\xi_1\otimes\cdots\otimes\xi_n=\langle\xi,\xi_1\rangle\xi_2\otimes\cdots\otimes
    \xi_n
\end{align}
for $\xi_i\in\hu, 1\leq i\leq n, n\geq 1.$ The following lemma is crucial for our purposes (see \cite[Lemma 2.2]{MS} and \cite[Theorem 6]{Boz} for proof):

\begin{lem}\label{lem:free annihilation operators are bounded}
The free annihilation operator $l(\xi)$  extends to a  bounded linear operator on $\FockTH$ with
\begin{align*}
    \|l(\xi)\|\leq \frac{\|\xi\|_U}{\sqrt{w(q)}},\;\;\;w(q)=(1-q^2)^{-1}\prod_{k=1}^\infty\frac{(1-q^k)}{(1+q^k)}
\end{align*}
where $q=\max_{i,j\in{N}}|q_{ij}|<1$.
\end{lem}
We next consider the $T$-{\em deformed creation and annihilation operators} on $\FockTH$, denoted $a^*(\xi)$ and $a(\xi)$ respectively
for $\xi\in \h$,  given by the assignments
\begin{align}\label{eq:definition of left creation and annihilation operators}
\begin{split}
&a^*(\xi)\Omega=\xi,\;\;a^*(\xi)(\xi_1\otimes\ldots\otimes\xi_n)=\xi\otimes\xi_1\otimes\cdots\xi_n,\;\;\;\text{and}\\
&a(\xi)\Omega=0,\;\; a(\xi)(\xi_1\otimes\cdots\otimes\xi_n)=l(\xi)(1+T_1+T_1T_2+\ldots+T_1\ldots T_{n-1})\xi_1\otimes\ldots\otimes\xi_n
\end{split}
\end{align}
for $\xi_1\otimes\cdots\otimes\xi_n\in\hu^{\otimes n}, n\geq 1$, 
where $l(\xi)$ is  the free annihilation operator as in \eqref{eq:free annihilation operator} and each $T_k\in \mathbf{B}(\hu^{\otimes n})$ is the operator as in \eqref{eq: notation for T_k} for $1\leq k\leq n$. The operators $a^*(\xi)$ and $a(\xi)$ extend to bounded linear maps on $\FockTH$ and they are adjoint to each other satisfying $\|a(\xi)\|=\|a^*(\xi)\|\leq \|\xi\|_U(1-q)^{-1/2}$ (see \cite{BS}). One easily verifies that if $\xi_{k}\in \hu^{(i_k)}$ for $i_k\in {N}$ and $1\leq k\leq n$, then 
\begin{align}\label{eq:annihilation operator acting on simple tensors}
    a(\xi)(\xi_n\otimes\cdots\otimes\xi_1)=\sum_{k=1}^n\langle\xi,\xi_k\rangle_U\prod_{t=1}^{n-k}q({i_k, i_{t+k}})\xi_n\otimes\cdots\otimes\widehat{\xi}_k\otimes\cdots\otimes\xi_1
\end{align}
where $\hat{\xi}_k$ means the absence of the vector $\xi_k$.
We adopt the following notations in order to minimize the length of calculations:
\begin{itemize}
    \item $\xi_1\cdots\xi_n=\xi_1\otimes\cdots\otimes\xi_n$ for any $\xi_i\in\hu, 1\leq i\leq n$.
    \item $C_q=\prod_{n=1}^\infty \frac{1}{1-q^n}$ for $0<q<1$.
    \item $[n]_q=1+q+\cdots+q^{n-1}$ for $n\geq 1$ and $[0]_q=1$.
    \item $[n]_q!=\prod_{k=1}^n[k]_q$ for $n\geq 1$ and $[0]_q!=1$.
\end{itemize}
We now mention an explicit description for the inner product of two simple tensors. 
  We could not trace anywhere a formula written in the following form, so we include the proof for completeness (also see \cite[Lemma 2.1]{BKM}).
\begin{lem}
For $\xi_k\in \hu^{(i_k)}$ and $\eta_k\in \hu^{(j_k)}$, $i_k, j_k\in {N}, 1\leq k\leq n$,  we have
\begin{align}\label{eq:formula for inner product of simple tensors}\begin{split}\langle\xi_n\cdots\xi_1, \eta_n\cdots\eta_1\rangle_T &=\sum_{\sigma\in S_n}\prod_{\underset{\sigma(u)>\sigma(v)}{1\leq u<v\leq n}}q(i_u, i_v)\prod_{t=1}^n\langle\xi_t, \eta_{\sigma(t)}\rangle_U\\
    &=\sum_{\sigma\in S_n}\prod_{\underset{\sigma(u)>\sigma(v)}{1\leq u<v\leq n}}q(j_{\sigma(u)}, j_{\sigma(v)})\prod_{t=1}^n\langle\xi_t, \eta_{\sigma(t)}\rangle_U\\
    &=\sum_{\sigma\in S_n}\prod_{\underset{\sigma(u)>\sigma(v)}{1\leq u<v\leq n}}q(j_u, j_v)\prod_{t=1}^n\langle\xi_{\sigma(t)}, \eta_{t}\rangle_U\\
    &=\sum_{\sigma\in S_n}\prod_{\underset{\sigma(u)>\sigma(v)}{1\leq u<v\leq n}}q(i_{\sigma(v)}, i_{\sigma(u)})\prod_{t=1}^n\langle\xi_{\sigma{(t) }}, \eta_{t}\rangle_U.
    \end{split}
\end{align}
\end{lem}
\begin{proof}
Note first that the last four expressions are indeed the same; for example to check that the first two expressions on the right side are same,  observe that  each term in the first sum of these is non-zero  only if $\langle\xi_t,\eta_{\sigma(t)}\rangle_U\neq0$ for all $1\leq t\leq n$ which is possible  only if $i_t=j_{\sigma(t)}$, in which case we have $q(i_u, i_v)=q(j_{\sigma(u)}, j_{\sigma(v)})$ for $1\leq v,u\leq n$. Similarly, the first sum and the last sum on the right are equal by noting that
\begin{equation}\label{eq:2.3}
\prod_{\underset{\sigma(u)>\sigma(v)}{1\leq u<v\leq n}}q(i_{u}, i_{v})=\prod_{\underset{\sigma^{-1}(r)>\sigma^{-1}(s)}{1\leq r<s\leq n}}q(i_{\sigma^{-1}(r)}, i_{\sigma^{-1}(s)})
\end{equation}
for any $\sigma\in S_n$. We shall prove equality of
the  left hand side with the first on the right. The 
    proof is by induction on $n$. For $n=1$ the equality is obvious. So assume them to be true for some  $n\geq 1$; then  calculate
    \begin{align*}
\langle&\xi_{n+1}\cdots\xi_1,\eta_{n+1}\cdots\eta_1\rangle_T=\langle\xi_{n+1}\cdots\xi_1,a^*(\eta_{n+1})\eta_n\cdots\eta_1\rangle_T\\
&=\langle a({\eta_{n+1}})( \xi_{n+1}\cdots\xi_1), \eta_{n}\cdots\eta_1\rangle_T\\
        &=\left\langle\sum_{k=1}^{n+1}\langle\eta_{n+1},\xi_{k}\rangle_U\prod_{t=1}^{n+1-k}q(i_k, i_{k+t})\xi_{n+1}\cdots\widehat{\xi}_k\cdots\xi_1, \eta_n\cdots\eta_1\right\rangle_T\quad\quad\quad(\text{by eq. }\eqref{eq:annihilation operator acting on simple tensors})\\
&=\sum_{k=1}^{n+1}\langle\xi_{k},\eta_{n+1}\rangle_U\prod_{t=1}^{n+1-k}q(i_k, i_{k+t})\left\langle\xi_{n+1}\cdots\widehat{\xi}_k\cdots\xi_1, \eta_n\cdots\eta_1\right\rangle_T\\
&=\sum_{k=1}^{n+1}\prod_{t=1}^{n+1-k}q(i_k, i_{k+t})\sum_{\sigma\in S_n}\prod_{\underset{\sigma(u)>\sigma(v)}{1\leq u<v\leq n}}q(i_{u}, i_{v})\;\langle\xi_k,\eta_{n+1}\rangle_U\left\langle\xi_{n+1}\cdots\widehat{\xi}_k\cdots\xi_1,\eta_{\sigma(n)}\cdots\eta_{\sigma(1)}\right\rangle_{\h_U^{\otimes n}}\\
&=\sum_{k=1}^{n+1}\prod_{t=1}^{n+1-k}q(i_k, i_{k+t})\sum_{\underset{\sigma(k)=n+1}{\sigma\in S_{n+1}}}\prod_{\underset{\sigma^{-1}(u)>\sigma^{-1}(v)}{1\leq u<v\leq n}}q(i_{\sigma^{-1}(u)}, i_{\sigma^{-1}(v)})\left\langle\xi_{n+1}\cdots\xi_1,\eta_{\sigma(n+1)}\cdots\eta_{\sigma(1)}\right\rangle_{\h_U^{\otimes {n+1}}}
    \end{align*}
    where in the last equality we have used \eqref{eq:2.3}, and for any fixed $k\in \{1,\ldots,n\}$ we have identified any permutation $\sigma \in S_n$ with a permutation, say $\sigma^{(k)}$, in $ S_{n+1}$ by declaring $\sigma^{(k)}(k)=n+1$  and keeping the order preserved, that is, 
    \begin{align*}
\sigma^{(k)}(t)= \left\{\begin{array}{ll}
         \sigma( t)& \mbox{ if }1\leq t<k,  \\
          n+1 & \mbox{ if }t=k,\\
         \sigma( t-1)+1&\mbox{ if }k<t\leq n+1.
      \end{array}  \right.\end{align*} 
      Note that the extra crossings that occur in the permutation $\sigma^{(k)}$ as compared to $\sigma$ are exactly those between the pairs $(k, n+1)$ and $(t, \sigma^{(k)}(t))$ for $t>k$; in such case we get from \eqref{eq:2.3} (and by setting $\tau=\sigma^{(k)}$) that
      \begin{align*}
\prod_{\underset{\tau(u)>\tau(v)}{1\leq u<v\leq n+1}}q(i_{u}, i_{v}) &=\prod_{\underset{\tau^{-1}(u)>\tau^{-1}(v)}{1\leq u<v\leq n+1}}q(i_{\tau^{-1}(v)}, i_{\tau^{-1}(u)})\\ &=\prod_{\underset{\tau^{-1}(u)>\tau^{-1}(v)}{1\leq u<v\leq n}}q(i_{\tau^{-1}(v)}, i_{\tau^{-1}(u)})\cdot\prod_{t=k+1}^{n+1}q(i_k, i_t).
      \end{align*}
     Finally we obtain
      \begin{align*}
\langle\xi_{n+1}\cdots\xi_1,\eta_{n+1}\cdots\eta_1\rangle_T&=  
\sum_{k=1}^{n+1}\sum_{\underset{\sigma(k)=n+1}{\sigma\in S_{n+1}}}\prod_{\underset{\sigma(u)>\sigma(v)}{1\leq u<v\leq n+1}}q(i_{u}, i_{v})\langle\xi_{n+1}\cdots\xi_1,\eta_{\sigma(n+1)}\cdots\eta_{\sigma(1)}\rangle_{\h_U^{\otimes {n+1}}}\\
&=\sum_{{\sigma\in S_{n+1}}}\prod_{\underset{\sigma(u)>\sigma(v)}{1\leq u<v\leq n+1}}q(i_{u}, i_{v})\langle\xi_{n+1}\cdots\xi_1,\eta_{\sigma(n+1)}\cdots\eta_{\sigma(1)}\rangle_{\h_U^{\otimes {n+1}}}
      \end{align*}
      which is exactly what was to be shown.
\end{proof}
For each $\xi\in\hr$, consider the operator $s(\xi)=a(\xi)+a^*(\xi)\in\mathbf{B}(\fth)$ and denote by $\Gamma_T(\hr, U_t)$ the von Neumann algebra generated by $\{s(\xi);\xi\in\hr\}$, which we call as {\em mixed $q$-Araki-Woods von Neumann algebra}. We will often denote this von Neumann algebra by $\M_T$ (or simply by $\mathsf{M}$ if $T$ is clear from the context). The vacuum vector $\Omega$ is a cyclic and separating vector for $\M_T$; so $\M_T$ is in standard form as an algebra acting on $\FockTH$. Let $\varphi$ denote the canonical normal faithful state $\langle\Omega,\cdot\Omega\rangle$ on $\vonT$. Denote by $L^2(\M_T)$ (or by $L^2(\M_T,\varphi)$) the GNS Hilbert space of $\M_T$ with respect to the state $\varphi$.  As mentioned above {$L^2(\M_T)$ is naturally identified with $\FockTH$}. Denote  by $\FockTHalg$ the subspace  spanned by finite simple tensors. Then $\FockTHalg$ is contained in $\M_T\Omega$ (see \cite{BKM}); hence for all $\xi\in\FockTHalg$, there is a unique operator $W(\xi)\in\M_T$ such that $W(\xi)\Omega=\xi$.  We refer the readers to \cite{BKM} for modular theory and other basic facts about the algebras $\M_T$.

\section{Dual and conjugate variables for mixed $q$-Araki-Woods algebras}\label{sec:dual variables}

\label{sec:dc}

In this section we will consider only the case of finite-dimensional $\Hil_\br$ and an orthogonal representation $U_t$ acting on $\hr$. Assume then that ${N}$ is a finite set and $\hr=\oplus_{i\in{N}}\hri$ where each $\hri$ is a $U_t$-invariant finite  dimensional real Hilbert space. Fix $-1<q(i,j)=q(j,i)<1$ for $i,j\in {N}$, let $q=\max_{i,j\in {N}}|q(i,j)|<1$ and consider the operator $T$ on $\mathbf{B}(\hu\otimes\hu)$ as defined in \eqref{eq:expression for T}. Assume that  $\dim{\hr}=d$. For this whole Section, we fix an orthonormal basis  $\{e_{1},\ldots,e_d\}$  of $\hu$ constructed in such a way that each $e_k\in \h_\C^{(j_k)}$ for some $j_k\in {N}$ (the basis can be chosen so as to contain eigenvectors of the analytic generator $A$ of $U_t$; see \cite{Sh1, Hiai}). For each $1\leq k\leq d$, set $A_k=W(e_k)$. Note that the set $\{A_k\}_{1\leq k\leq d}$ is  algebraically free   such that the unital $*$-subalgebra  $\C[A_1,\ldots,A_d]$ generated by the set is strongly dense in   $\M_T=\vonT$. 

We say that a tuple  $(D_1,\ldots,D_d)$ of unbounded operators on $\FockTH$ with $\FockTHalg$ contained in their domains and $\Omega$ contained in the domain of their adjoints is a {\em (normalized) dual system} for $(A_1, \ldots, A_d)$ if for all $k,\ell  \in \oned$ 
\[{[D_k ,A_\ell ]}=\langle\bar{e}_\ell , e_k\rangle_U P_{\bc \Omega}=\varphi(A_\ell A_k)P_{\bc \Omega}\;\;\mbox{ and } D_k \Omega=0,\]
where $P_{\C\Omega}$ is the projection onto the one-dimensional subspace $\C\Omega$, and as before $\Bar{e}_\ell$ is the complex conjugate of $e_\ell$. The existence of dual system ensures that each  $D_k$ is a closable operator on $\FockTH$.

 Next recall that if $\mathsf{B}\subseteq\M_T$ is a $\ast$-subalgebra, then a \emph{derivation} is a $\C$-linear map $\delta:\mathsf{B}\to\M_T\otimes\M_T^{\op}$ which satisfies the Leibniz rule: $\delta(ab)=a\delta(b)+b\delta(a)$ for all $a,b\in\mathsf{B}$. Here $\M_T^{\op}$ denotes the opposite algebra equipped with the natural state $\varphi^{\op}$. Then the {\em quasi-free difference quotients} $\partial_{k}$ are defined as unique derivations from $\mathbb{C}[A_1,\dots, A_d]$ into $\M_T \overline{\otimes} \M_T^{\op}$ such that \[\partial_k(A_\ell ) := \varphi(A_\ell A_k)\mathds{1}\otimes\mathds{1}\]
 for all $k,\ell \in \oned$, where $\mathds{1}$ is the identity of $\M_T$. The \emph{conjugate variable} will be a vector $\xi_k \in L^{2}(\M_T)$ such that 
\[
\langle \xi_k, x\mathds{1}\rangle_{L^2(\M_T)} = \langle \mathds{1}\otimes\mathds{1}, \partial_k(x)\mathds{1}\otimes \mathds{1}\rangle_{L^2(M_T\otimes \M_T^{\op})} 
\]
for all $x \in \C[A_1,\ldots, A_d]$. Remark that the existence of the conjugate variable is equivalent to $\mathds{1}\otimes\mathds{1}\in \dom(\partial_k^*)$ when $\partial_k$ is considered an (unbounded) operator from $L^2(\M_T)$ to $L^2(\M_T\overline{\otimes}\M_T^{\op})$; in this case it is given precisely by $\xi_k=\partial_k^*(\mathds{1}\otimes\mathds{1})$.   We also recall that the existence of dual variables implies existence of conjugate variables:

\begin{prop}[{See \cite[Theorem 2.5]{MS}, \cite[Proposition 3.1]{KSW}}]
Suppose that $(D_1,\dots, D_d)$ is a normalized dual system for $(A_1,\dots, A_d)$. Then $(D_{1}^{\ast}\Omega,\dots, D_{d}^{\ast}\Omega)$ are conjugate variables for $(A_1,\dots, A_d)$.
\end{prop}

\subsection*{Existence of Dual variables}
We now show the existence of a system of dual variables for the tuple $(A_1,\ldots,A_d)$ of operators corresponding to a fixed orthonormal basis $\{e_1,\ldots,e_d\}$ of $\hu$. We follow the lines of investigation as in \cite{MS, KSW}. In order to keep this article self-contained, we repeat here some of the notions considered in \cite{MS}. Firstly we recall below the rules  of drawing partitions of the set $\{0,1,\ldots,n\}$ consisting of singletons and pairs:

\begin{enumerate}
    \item Consider $n+1$ vertices $n>n-1>\cdots>1>0$.
    \item $0$ must be connected to some $k\in \{1,\ldots,n\}$ with height $1$.
    \item $\ell\in \{1,\ldots,k-1\}$ must be coupled with one of $\{k+1,\ldots,n\}$ with height $\ell+1$.
    \item Vertices which are not coupled with $\{1,\ldots,k-1\}$ should be singletons and are drawn with straight lines to the top.
\end{enumerate}
Define $B(n+1)$ as the set of partitions that satisfy the above rules. For $\pi\in B(n+1)$,  denote by $p(\pi)$ the set of pairings in $\pi$, and by $s(\pi)$ the set of singletons in $\pi$. Also for any $k\in \{0,\ldots, n\}$, write $\pi(k)=\ell$ if $(k,\ell )\in p(\pi)$ and $\pi(k)=k$ if $k\in s(\pi)$.

  We set some notations. Write $[d]=\{1,\ldots, d\}$ and 
  let $[d]^*$ denote the set of words in $[d]$. Write
  \begin{equation*}
      e_w=e_{\alpha_1}\cdots e_{\alpha_n}=e_{\alpha_1}\otimes\ldots \otimes e_{\alpha_n}\;\;\mbox{ and }\;\; e_0=e_{()}=\Omega,
  \end{equation*} for any word $w=\alpha_1\cdots \alpha_n\in [d]^*$ of length $n\geq1$ and empty word $()$. The set $\{e_w; w\in [d]^*\}$ is linearly independent and spans $\FockTHalg$.
Observe that for any $\alpha\in [d]$ and for any word $\alpha_n\cdots\alpha_1\in [d]^*$ of length $n\geq1$,  the expression $A_\alpha=W(e_\alpha)=a^*(e_\alpha)+a(\bar{e}_\alpha)$  yields
 \begin{align}\label{eq: formula for A_alpha}
 \begin{split}
    A_\alpha(e_{\alpha_n \ldots \alpha_1} )&=e_\alpha e_{\alpha_n\cdots\alpha_1}+\sum_{k=1}^{n}\langle \bar{e}_\alpha, e_{\alpha_k}\rangle_U\prod_{k<t\leq n}q(i_k, i_{t})e_{\alpha_n\cdots\widehat{\alpha}_k\ldots \alpha_1}\\
     &=e_\alpha e_{\alpha_n\cdots\alpha_1}+\sum_{k=1}^{n}\langle \bar{e}_\alpha, e_{\alpha_k}\rangle_U\prod_{t=1}^{n-k}q(i_k, i_{n-t+1})e_{\alpha_n\cdots\widehat{\alpha}_k\ldots \alpha_1} 
     \end{split}
 \end{align}
where $i_k\in{N}$ is such that $e_{\alpha_k}\in \h_\C^{(i_k)}$ for $1\leq k\leq n$.

Taking cue from the approach of \cite{MS} and \cite{KSW}, we now establish the existence of dual/conjugate variables in the non-tracial context of mixed $q$-Araki-Woods von Neumann algebras.

\begin{prop}\label{prop: Algebraic formula for dual variable D_alpha}
	The algebraic formula for the tuple  $(D_1,\ldots,D_d)$ of the dual variables of $(A_1,\ldots, A_d)$ is given as follows; for fixed $\alpha\in [d]$ and $w=\alpha_n\cdots\alpha_1\in [d]^*$ such that $ e_\alpha\in \h_\C^{(i_0)}$ 
 and $e_{\alpha_k}\in \hc^{(i_k)}$, for some $i_0, i_k\in {N}$, $1\leq k\leq n$:
	\begin{equation}\label{eq:expression for D_alpha}
 D_\alpha\Omega=0,\;\;\;\;\;D_\alpha (e_{\alpha_n \cdots \alpha_1}) = \sum_{\pi \in B(n+1)} (-1)^{{\pi(0)-1}} Q(\pi) E(\pi) e_{s(\pi)}, 
 \end{equation}
	where $s(\pi)$  is the set of singletons of $\pi$,  and
 \begin{align*}
&Q(\pi):
=\prod_{0\leq v< u< \pi(0)}q(i_v, i_u)\cdot\prod_{\underset{\pi(v)>\pi(u)}{0< v<u< \pi(0)} } q(i_{\pi(v)}, i_{\pi(u)})
\cdot\prod_{\underset{v\in s(\pi), u\notin s(\pi)}{\pi(0)<v<u\leq n}}q(i_v, i_u)
\\  &E(\pi):= \prod_{\underset{{k>\ell}}{(k,\ell)}\in p(\pi)}\langle \bar{e}_{\alpha_k}, e_{\alpha_\ell}\rangle_U\\
&e_{s(\pi)}:=e_{\alpha_{\ell_s}\cdots \alpha_{\ell_1}}=e_{\alpha_{\ell_s}}\otimes\cdots\otimes e_{\alpha_{\ell_1}}\;\mbox{ for } s(\pi)=\{\ell_s>\ldots>\ell_1\}
 \end{align*}
 with the prescription that $\alpha_0=\alpha$.
\end{prop}
\begin{rem}\label{rem: different formulas for Q(pi)}We have divided the terms of $Q(\pi)$ into three products. The first product (resp. the second product) corresponds to the crossings in the partition occurring on the line of the vertex $`u$' (resp. $`\pi(u)'$) for $0<u< \pi(0)$,  with the lines of another pair. 
While the third product corresponds to the crossings on the lines of the singletons. Note that $E(\pi)$ is non-zero  only if $i_k=i_{\pi(k)}$ for any pair $(k,\pi(k))\in p(\pi)$ (because $\xi\in \h_\C^{(i)}$ if and only if $\bar{\xi}\in \h_\C^{(i)}$ and $\hc^{(i)}\perp\hc^{(j)}$ if $i\neq j$). Thus in such cases, in the expression for $Q(\pi)$,  the term $q(i_v,i_u)$ can be replaced by $q(i_{\pi(v)}, i_{\pi(u)})$  for any indices $u,v$ which  are part of the pairings of $\pi$ and then we use symmetry of the matrix $(q(i,j))$ to get the following:
\begin{align}
  \begin{split} E(\pi) Q(\pi)&=E(\pi)\prod_{{0\leq v< u<\pi(0)}}q(i_{\pi(v)}, i_{\pi(u)})\cdot\prod_{\underset{\pi(v)>\pi(u)}{0< v<u< \pi(0)} } q(i_{\pi(v)}, i_{\pi(u)})
\cdot\prod_{\underset{v\in s(\pi), u\notin s(\pi)}{\pi(0)<v<u\leq n}}q(i_v, i_u)\\
&=E(\pi)\prod_{{0\leq v< u<\pi(0)}}q(i_{v},i_{u})\cdot\prod_{\underset{\pi(v)>\pi(u)}{0< v<u< \pi(0)} } q(i_{v}, i_{u})
\cdot\prod_{\underset{v\in s(\pi), u\notin s(\pi)}{\pi(0)<v<u\leq n}}q(i_v, i_u).
\end{split}
\end{align}
\end{rem}
\begin{proof}
Consider the (unbounded) operator $D_\alpha$  on $\FockTH$ defined as in \eqref{eq:expression for D_alpha} on vectors which  span   the subspace  $\FockTHalg$. We will show that $[D_\alpha, A_\beta]=\langle \bar{e}_\beta, e_\alpha\rangle_UP_{\C\Omega}$ on  $\FockTHalg$ for all $\beta\in[d]$. It is straightforward to	check that 
	\[ [D_\alpha , A_\beta] \Omega=D_\alpha A_\beta\Omega =D_\alpha  e_\beta = \langle \bar{e}_\beta, e_\alpha\rangle_U\Omega. \]
	Then for $\alpha_{n+1}\in [d]$, use \eqref{eq: formula for A_alpha} to compute 
	\begin{align*}	
 A_{\alpha_{n+1}} &D_\alpha (e_{\alpha_n  \ldots \alpha_1} )  =  \sum_{\sigma \in B(n+1)} (-1)^{\sigma(0)-1} Q(\sigma) E(\sigma) A_{\alpha_{n+1}}e_{s(\sigma)}\\ &=\sum_{\sigma \in B(n+1)} (-1)^{\sigma(0)-1} Q(\sigma) E(\sigma) \;{e}_{\alpha_{n+1}}e_{s(\sigma)}
		\\ &\quad+ \sum_{\sigma \in B(n+1)} \sum_{v=1}^{|s(\sigma)|} (-1)^{\sigma(0)-1} Q(\sigma) E(\sigma) \langle \bar{e}_{\alpha_{n+1}} ,{e}_{\alpha_{s(\sigma)_v}}\rangle_U \prod_{\underset{{t\in s(\sigma)}}{s(\sigma)_v<t\leq n}}q(i_{s(\sigma)_v}, i_t)\;    e_{s(\sigma)\setminus s(\sigma)_v}
	\end{align*}
 where $e_{s(\sigma)\setminus s(\sigma)_v}$ means that we omit the term $e_{\alpha_{s(\sigma)_v}}$ from $e_{s(\sigma)}$.
Next we compute 
 \begin{align*}
		D_\alpha  &A_{\alpha_{n+1}} (e_{\alpha_n  \ldots \alpha_1}) \\
  &= D_\alpha  \left(e_{\alpha_{n+1}\cdots\alpha_1}  + \sum_{u=1}^n \langle \bar{e}_{\alpha_{n+1}} , e_{\alpha_u}\rangle_U \prod_{u<t\leq n} q(i_{u},i_{t})  e_{\alpha_n\cdots \widehat{{\alpha}}_u \cdots \alpha_1}  \right)
		\\&= D_\alpha  e_{\alpha_{n+1}  \cdots \alpha_1}  \\& \quad\quad+ \sum_{u=1}^n      \sum_{\pi \in B(n)}  (-1)^{\pi(0)-1} Q(\pi,u)E(\pi, u) \langle \bar{e}_{\alpha_{n+1}} , e_{\alpha_u}\rangle_U\prod_{u<t\leq n} q(i_{u},i_t)  \;e_{s(\pi,u)}
	\end{align*}
where $e_{s(\pi, u)}$, $Q(\pi, u)$ and $E(\pi, u)$ respectively denote  $e_{s(\pi)}$, $Q(\pi)$  and  $E(\pi)$ only with the emphasis that  $\pi\in B(n)$ acts on the $n-1$ letters $\alpha_n,\ldots,\widehat{\alpha}_u, \ldots,\alpha_1$.	
	As in the proof of \cite[Proposition 4.5]{MS} we identify all terms in the last  sum of $A_{\alpha_{n+1}} D_\alpha (e_{\alpha_n \ldots \alpha_1} ) $ with  some terms  in the last sum of $D_\alpha  A_{\alpha_{n+1}} (e_{\alpha_n\cdots \alpha_1}) $ as follows: take a pair $(\sigma,v)$ with $\sigma\in B(n+1)$ and $v\in \{1,\ldots,|s(\sigma)|\}$, then  get a pair  $(\pi, u)$ where $u=s(\sigma)_v$ and  $\pi \in B(n)$ is a partition of the set $\{0,\ldots, n\}\setminus\{u\}$ obtained from $\sigma$ in an order preserving fashion  by removing the singleton $s(\sigma)_v$. Then observe  that $\sigma(0)=\pi(0)$ and  the pairings in $p(\sigma)$ and $p(\pi)$ remain the same while the singleton sets satisfy $s(\sigma)=s(\pi)\cup\{s(\sigma)_v\}$. This means that $e_{s(\pi, u)}=e_{s(\sigma)\setminus s(\sigma)_v}$ and $E(\pi, u)=E(\sigma)$ so that 
	\[  E(\pi,u) \langle \bar{e}_{\alpha_{n+1}} , e_{\alpha_u}\rangle_U e_{s(\pi, u)} = E(\sigma) \langle \bar{e}_{\alpha_{n+1}} , e_{\alpha_{s(\sigma)_v}}\rangle_U e_{s(\sigma)\setminus s(\sigma)_v}. \]
Moreover the extra terms which appear in the product of $Q(\sigma)$ as compared to that of $Q(\pi, u)$ are  the ones due to the crossing points on the line of the singleton $s(\sigma)_v$. Such crossings  occur with those pairings in $\sigma$ which have one of the vertices on the left of $u=s(\sigma)_v$. Therefore  we get
\begin{equation*}
    {Q(\sigma)}={Q(\pi,u)}\cdot\prod_{\underset{t\notin s(\sigma)}{s(\sigma)_v<t\leq n}}q(i_{s(\sigma)_v},i_t)
=Q(\pi, u)\cdot\frac{\prod_{{s(\sigma)_v<t\leq n}}q(i_{s(\sigma)_v}, i_{t})}{\prod_{\underset{t\in s(\sigma)}{s(\sigma)_v<t\leq n}}q(i_{s(\sigma)_v}, i_{t})}
\end{equation*}
where $\frac{0}{0}$ is  understood to be $0$; this implies by using $u=s(\sigma)_v$ that
\begin{align*}
Q(\sigma)\cdot\prod_{\underset{t\in s(\sigma)}{s(\sigma)_v<t\leq n}}q(i_{s(\sigma)_v}, i_{t})=Q(\pi,u)\cdot\prod_{{u< t\leq n}}q(i_u, i_{t}).
\end{align*}
This shows that the term corresponding to $\sigma$ and $v$ in the second sum of  $A_{\alpha_{n+1}} D_\alpha (e_{\alpha_n  \ldots \alpha_1} ) $  appears as the term corresponding to $\pi$ and $u$ in the second sum of 	$D_\alpha  A_{\alpha_{n+1}} (e_{\alpha_n \ldots \alpha_1})$, and at the same time note that this correspondence exhausts all the  pairs $(\pi, u)$ with $\pi\in B(n)$ and $u>\pi(0)$. Thus  after the subtracting we get the following:
	\begin{align}\label{eq:fomula after substractinf thw two calculation}
 \begin{split}
 - [D_\alpha , A_{\alpha_{n+1}}]&e_{\alpha_n \ldots \alpha_1} +
  D_\alpha  e_{\alpha_{n+1}  \cdots \alpha_1} \\&=\sum_{\sigma \in B(n+1)} (-1)^{\sigma(0)-1} Q(\sigma) E(\sigma) \;e_{\alpha_{n+1}}  e_{s(\sigma)}  \\
  &\quad-\sum_{u=1}^n      \sum_{\underset{\pi(0)\geq u}{\pi \in B(n)}}  (-1)^{\pi(0)-1} Q(\pi,u) E(\pi, u) \langle \bar{e}_{\alpha_{n+1}} , e_{\alpha_u}\rangle_U\prod_{u<t\leq n} q(i_{u},i_t) e_{s(\pi,u)}.
  \end{split}
  \end{align}
Now we try to understand the terms in the last sum above. Recall that each pair  $(\pi, u)$ with  $\pi\in B(n)$ and $u\in \{1,\ldots, n\}$ acts on  the $n-1$ letters $\alpha_n,\ldots, \widehat{\alpha}_u, \ldots, \alpha_1$. As in the second step of the proof of \cite[Proposition 4.5]{MS},   to each such pair $(\pi, u)$ with $u\leq \pi(0)$, we associate a new partition $\pi'\in B(n+2)$  as follows: 
construct the partition $\pi'\in B(n+2)$ of $\{1,\ldots, n+1\}$ out of $(\pi, u)$   by inserting a new point at $u$ and pairing it with $n+1$ with the requirement that  $\pi'\setminus(n+1, u)=\pi$. 
We then have
\begin{align*}
&\pi'(0)=\pi(0)+1.
\end{align*}
Observe that the set of singletons remains the same in both $\pi$ and $\pi'$  while the set of pairings satisfies $p(\pi')=p(\pi)\cup\{(n+1, u)\}$; hence we get 
\begin{equation*}
    e_{s(\pi')}=e_{s(\pi, u)}\;\;\mbox{ and }\;\;E(\pi')=E(\pi,u)\;\langle \bar{e}_{\alpha_{n+1}}, e_{\alpha_u}\rangle_U.
\end{equation*} 
Further note that  there are some extra crossing points for the partition $\pi'$ as compared to $(\pi, u)$  due to the coupling $(n+1, u)$; such crossings occur  on the line of $t$ of the couplings $(t, \pi'(t))$ for $0\leq t<u$, on  the lines of both $t$ and $\pi(t)$ of the couplings $(t,\pi'(t))$  for $\pi'(0)> t>u$,  and  with the singletons of $\pi'$. Therefore  we obtain (whenever $E(\pi')\neq 0$) that
\begin{align*}
    \nonumber Q(\pi')&={Q(\pi, u)}\cdot\prod_{0\leq t<u}q(i_t, i_u)\cdot\prod_{{u<t<\pi'(0)}}q(i_t, i_u)\cdot\prod_{u<t< \pi'(0)}q(i_{\pi'(u)}, i_{\pi'(t)})\cdot\prod_{t\in s(\pi')}q(i_t, i_{n+1})\\
    &={Q(\pi, u)}\prod_{0\leq t<u}q(i_{\pi'(t)}, i_u)\cdot\prod_{{u<t<\pi'(0)}}q(i_{\pi'(t)}, i_u)\cdot\prod_{u<t< \pi'(0)}q(i_{u}, i_{\pi'(t)})\cdot\prod_{t\in s(\pi')}q(i_t, i_u)
\end{align*}
 where to get the last expression, we have used the symmetry of the matrix $(q(i,j))$ and the same argument as in Remark \ref{rem: different formulas for Q(pi)} that whenever $E(\pi')\neq0$, then $i_t=i_{\pi'(t)}$ for any coupling $(t, \pi'(t))$  so that we can replace any $q(i_t, i_r)$ by $q(i_{\pi(t)}, i_r)$, $r\in{N}$.  On the other hand, if $E(\pi')=0$ then $E(\pi, u)=0$, so the corresponding term does not contribute in the sum. A careful observation of the last product above (and again the symmetry of the matrix $(q(i,j))$) then yields
\begin{equation}\label{eq:expression for Q(pi')}
    Q(\pi')=Q(\pi, u)\cdot \prod_{u<t\leq n}q(i_u, i_t).
\end{equation}
Therefore the expression in \eqref{eq:fomula after substractinf thw two calculation} can finally be rewritten in the following way:  
	\begin{align*}
	    - [D_\alpha , A_{\alpha_{n+1}}]&(e_{\alpha_n  \ldots \alpha_1} )+
  D_\alpha  e_{\alpha_{n+1}}  \ldots e_{\alpha_1} \\&=\sum_{\sigma \in B(n+1)} (-1)^{\sigma(0)-1} Q(\sigma) E(\sigma) \;e_{\alpha_{n+1}}e_{s(\sigma)}  \\
  &\;\;\;\;\;\;\;\;\;\;+    \sum_{\underset{\pi'(n+1)\mbox{ is not a singleton}}{\pi' \in B(n+2)}}  (-1)^{\pi'(0)-1} Q(\pi') E(\pi') \;e_{s(\pi')}\\
  &=\sum_{{\pi' \in B(n+2)}}  (-1)^{\pi'(0)-1} Q(\pi') E(\pi') \;e_{s(\pi')}
	\end{align*}
 where the last equality follows because each partition $\sigma\in B(n+1)$ gets identified exactly to a partition $\sigma'\in B(n+2)$ in an order preserving fashion such that $n+1$ is a singleton in $\sigma'$; in such  case $\sigma(0)=\sigma'(0), Q(\sigma)=Q(\sigma'), E(\sigma)=E(\sigma')$ and $s(\sigma')=\{n+1\}\cup s(\sigma)$ so that $e_{s(\sigma')}=e_{\alpha_{n+1}} e_{s(\pi)}$. We have thus shown that 
	 $[D_\alpha , A_{\alpha_{n+1}}]e_{\alpha_{n}\cdots \alpha_1} =0$. This completes the proof.
\end{proof}

We now give the main result of this  Section in the following theorem. Note that since $\pi(n+1)$ has at least one singleton whenever $n$ is even, $\langle\Omega, D_\alpha e_w\rangle=0$ for any word $w$ of even length and $\alpha\in [d]$. Also if $n=2m-1$ for $m\geq 1$, then in the summation of $D_\alpha$, only those $\pi\in B(2m)$ contribute a non-zero term in the inner product $\langle\Omega, D_\alpha e_w\rangle$ which has no singleton; this occurs only if $\pi(0)=m$. In this case, we have
\begin{align}\label{eq:inner product of D_alpha}
    \langle\Omega, D_\alpha e_\omega\rangle=\sum_{\underset{\pi(0)=m}{\pi\in B(2m)}}(-1)^{m-1}Q(\pi, \omega)E(\pi, \omega)
\end{align}
where $Q(\pi, \omega)$ and $E(\pi, \omega)$ are written respectively for $Q(\pi)$ and $E(\pi)$ in order to make the dependency on $\omega$ explicit.

\begin{tw} \label{thm:existence of conjugate variables}
	For each $\alpha\in[d]$ we have $\Omega \in \dom D_\alpha^*$. Thus $(D_{1}^{\ast}\Omega,\dots, D_{d}^{\ast}\Omega)$ forms a set of conjugate variables for $(A_{1},\dots, A_{d})$.
\end{tw}
\begin{proof}
Fix $\alpha\in [d]$ and let $e_\alpha\in \h_\C^{(i_\alpha)}$ for some $i_\alpha\in {N}$.	As in  \cite[Theorem 4.6]{MS} and \cite[Proposition 1.4]{KSW} we will prove that the functional $\langle \Omega, D_\alpha(\cdot)\rangle_T$ on $\FockTHalg$ extends to a bounded  map on $\FockTH$. Fix then an element $\sum_{w\in [d]^*}\lambda_we_w$ in $\FockTHalg$ where the scalars $\lambda_w\in \C$ are  $0$ except for finitely many, and then calculate using \eqref{eq:inner product of D_alpha}:
\begin{align}\label{eq:expression for the inner product of omegea, D_alpha omega}
     \langle \Omega, D_\alpha (\sum_{w \in [d]^*} \lambda_w e_w)\rangle_T=\sum_{m=1}^\infty \sum_{\underset{\pi(0)=m}{\pi \in B(2m)}\;\;} \sum_{|w|=2m-1} \lambda_w (-1)^{m-1} Q(\pi, w)E(\pi, w). 
     \end{align}
Fix  $m\geqslant 1$ and  rewrite each word $w$ of length $2m-1$ as 
	$w_1\beta w_1'$ with $w_1,w_1'$ being words of length $m-1$ and $\beta  \in [d]$; then identify each $\pi \in B(2m)$, $\pi(0)=m$ with a permutation $\rho:=\rho_\pi\in S_{m-1}$ requiring that the couplings satisfy $(t, m+\rho(t))\in p(\pi)$ for $1\leq t\leq m-1$. Under such identification  write $Q(\rho, w_1, w_1',\beta)$ and $E(\rho, w_1,w_1',\beta)$ respectively for $Q(\pi, w)$ and $E(\pi, w)$; which can further be rewritten   as follows: 
 if  $w_1=\alpha_{m-1}\cdots\alpha_1$,  $w_1'=\alpha_{m-1}'\cdots\alpha_1' $ and  $e_{\alpha_{t}}\in \hc^{(i_t)}$, $e_{\alpha'_t}=\h_\C^{(i_t')}$, $e_\beta\in \hc^{(j)}$ for some $j,i_t, i'_t\in {N}$,  $1\leq t\leq m-1$, then 
 \begin{align*}
 &E(\rho, w_1, w_1', \beta) = \langle \bar{e}_\beta, e_\alpha\rangle_U\prod_{t=1}^{m-1} \langle \bar{e}_{\alpha_{\rho(t)}}, e_{\alpha'_t}\rangle_U=\langle \bar{e}_\beta, e_\alpha\rangle_U\prod_{t=1}^{m-1}\langle \bar{e}_{\alpha_t}, e_{\alpha'_{\rho^{-1}(t)}}\rangle_U\\
&Q(\rho, w_1,w_1',\beta)
=\prod_{u=1}^{m-1}q(i_\alpha,i_{u}')\cdot\prod_{{1\leq v< u\leq m-1}}q(i_{v}', i_{u}')\cdot\prod_{\underset{\rho(v)>\rho(u)}{1\leq v<u\leq m-1}}  q(i_{\rho(v)}, i_{\rho(u)}).
\end{align*}
Moreover, whenever $E(\pi, w)\neq 0$, then we use the same arguments as in Remark \ref{rem: different formulas for Q(pi)} to get the following expression:
\begin{equation*}
   Q(\rho, w_1, w_1',\beta)= \prod_{t=1}^{m-1}q(i_\alpha,i_{\rho(t)})\cdot\prod_{{1\leq v< u\leq m-1}}q(i_{\rho(v)}, i_{\rho(u)})\cdot\prod_{\underset{\rho(v)>\rho(u)}{1\leq v<u\leq m-1}}  q(i_{\rho(v)}, i_{\rho(u)})
\end{equation*}
which can further be written  as follows by rearranging the indices of the first two products  and then by using the symmetry of the matrix $(q(i,j))$:
\begin{equation*}
    Q(\rho, w_1, w_1',\beta)=\prod_{t=1}^{m-1}q(i_\alpha,i_{t}) \prod_{{1\leq v< u\leq m-1}}q(i_{v}, i_{u})\cdot\prod_{\underset{\rho(v)>\rho(u)}{1\leq v<u\leq m-1}}  q(i_{\rho(v)}, i_{\rho(u)}).
\end{equation*}
Therefore the sum in \eqref{eq:expression for the inner product of omegea, D_alpha omega} can be rewritten as follows:
	\begin{align*} 
	 \langle \Omega, D_\alpha& (\sum_{w \in [d]^*} \lambda_w e_w)\rangle_T	
		\\&= \sum_{m=1}^\infty(-1)^{m-1}  \sum_{\beta \in [d]} \sum_{\rho \in S_{m-1}\;\;} \sum_{|w|=m-1} \sum_{|w'|=m-1} 
		\lambda_{w\beta w'} \; Q(\rho, w, w',\beta)\;  E(\rho, w, w',\beta) \end{align*}
 We set the following notations: For each word $w=\alpha_n\cdots\alpha_1$ with $e_{\alpha_t}=\hc^{(i_t)}$ for $i_t\in{N}, 1\leq t\leq n$, write $i_t=i_t^{(w)}$ in order to show dependency on $w$. Also write $\bar{e}_w=\bar{e}_{\alpha_n}\otimes\cdots\otimes\bar{e}_{\alpha_1}$ and $\rho(w)=\alpha_{\rho(n)}\cdots\alpha_{\rho(1)}$ for any permutation $\rho\in S_{n}$. In this notation, we note that $E(\rho, w_1, w_1',\beta)=\langle\bar{e}_\beta,e_\alpha\rangle_U\langle e_{\rho(w)}, e_{w'}\rangle_{\hu^{\otimes^n}}$. We now calculate
	\begin{align*} 
		& \sum_{\beta \in [d]} \sum_{\rho \in S_{m-1}} \sum_{|w|=m-1} \sum_{|w'|=m-1} 
		\lambda_{w\beta w'} \; Q(\rho, w,w',\beta) \; E(\rho, w, w',\beta)
		\\&=\sum_{\beta \in [d]}  
		\sum_{|w|=m-1} \sum_{|w'|=m-1}\prod_{{1\leq v<u\leq m-1}}q(i^{(w)}_{v}, i^{(w)}_u)\prod_{t=1}^{m-1}q(i_\alpha,i_{t}^{(w)})\lambda_{w\beta w'}\langle \bar{e}_\beta, e_\alpha\rangle_U\\
  &\quad\quad\quad\cdot\sum_{\rho \in S_{m-1}} \prod_{\underset{\rho(t)>\rho(r)}{1\leq t<r\leq m-1}}  q(i^{(w)}_{\rho(r)}, i^{(w)}_{\rho(t)})\prod_{t=1}^{m-1}\langle \bar{e}_{w_{\rho(t)}}, e_{w'_t}\rangle_U
		\\&= \sum_{\beta \in [d]}  
		\sum_{|w|=m-1} \sum_{|w'|=m-1}\prod_{{1\leq v< u\leq m-1}}q(i^{(w)}_{v}, i^{(w)}_u)\prod_{t=1}^{m-1}q(i_\alpha,i_{t}^{(w)})\lambda_{w\beta w'}\langle \bar{e}_\beta, e_\alpha \rangle_U\langle \bar{e}_w,e_{w'}\rangle_T\\
  &= \sum_{\beta \in [d]}\sum_{|w|=m-1}\prod_{{1\leq v< u\leq m-1}}q(i^{(w)}_{v}, i^{(w)}_u) \prod_{t=1}^{m-1}q(i_\alpha,i_{t}^{(w)})\langle \bar{e}_\beta, e_\alpha \rangle_U {\left\langle \bar{e}_w,\sum_{|w'|=m-1}\lambda_{w\beta w'}\;e_{w'}\right\rangle}_T.
	\end{align*}
 Now denote  by ${L}_{w\beta}$ the composition of the $m$ relevant  free annihilation operators  i.e. $L_{w\beta}=l(e_{\beta})l(e_{\alpha_{1}})\cdots l(e_{\alpha_n})$ for any word $w=\alpha_n\cdots\alpha_1$, where $l(e_\gamma)$ is as defined in \eqref{eq:free annihilation operator} for any $\gamma\in [d]$. 
 Recall our assumption that the set $\{e_\gamma\}_{\gamma\in[d]}$ is an orthonormal basis for $\h_U$ so that $L_{w\beta}(e_{w''\gamma w'})=0$ for  words $w', w'' $ of length $m-1$ and $\gamma\in [d]$ such that $w''\gamma\neq  w\beta$. Then for fixed $w, \beta,\alpha$ as above we obtain
 \begin{align*}
     \sum_{|w'|=m-1}\lambda_{w\beta w'}
     \;e_{w'}&= 
     \sum_{\gamma\in [d]}\sum_{|w''|=m-1}\sum_{|w'|=m-1}\lambda_{w''\gamma w'}\;{L}_{w\beta}(e_{w''\gamma w'})\\
     &= {L}_{w\beta}\left(\sum_{\gamma\in [d]}\sum_{|w''|=m-1} \sum_{|w'|=m-1}\lambda_{w''\gamma w'}e_{w''\gamma w'}\right)\\
     &={L}_{w\beta} \left(\sum_{|w''|=2m-1}\lambda_{w''}e_{w''}\right)
 \end{align*}
 so that 
 \begin{align*}
    \langle \bar{e}_w,  \sum_{|w'|=m-1}\lambda_{w\beta w'}
     \;e_{w'}\rangle_T=\langle \bar{e}_w,   {L}_{w\beta} \left(\sum_{w''\in [d]^*}\lambda_{w''}e_{w''}\right)\rangle_T.
 \end{align*} 
which follows by noting that the operator ${L}_{w\beta}$ maps ${e}_{w'}$ with $|w'|=2m'-1$ to the subspace spanned by $\{e_v\}_{|v|=m-1}$ and hence $\langle \Bar{e}_w, {L}_{w\beta}(e_{w'})\rangle_T=0$ if $m\neq m'$ and $|w|=m-1, |w'|=2m'-1$. Thus we get
\begin{align*}
    &\sum_{\beta \in [d]}\sum_{\rho \in S_{m-1}} \sum_{|w|=m-1}  \sum_{|w'|=m-1} 
		\lambda_{w\beta w'} \; Q(\rho, w, w',\beta) \; E(\rho, w, w',\beta)\\
  &=\sum_{\beta\in [d]}\sum_{|w|=m-1}\prod_{{1\leq v< u\leq m-1}}q(i^{(w)}_{v}, i^{(w)}_u)\prod_{t=1}^{m-1}q(i_\alpha,i_{t}^{(w)})\langle\Bar{e}_\beta, e_\alpha\rangle_U\left\langle \bar{e}_w, {L}_{w\beta}\left(\sum_{w''\in[d]^*}\lambda_{w''}e_{w''}\right)\right\rangle_T
\end{align*}
 If we set 
\begin{align}\label{eq:expression for xi_malpha}\xi_{m,\alpha}=\sum_{\beta\in [d]}\sum_{|w|=m-1}\prod_{{1\leq v<u\leq m-1}}q(i^{(w)}_{v}, i^{(w)}_u)\prod_{t=1}^{m-1}q(i_\alpha,i_{t}^{(w)}) \langle e_\alpha,\Bar{e}_\beta\rangle_U  {L}_{w\beta}^*(\bar{e}_w)\end{align}
 then we have
\begin{align*}
    \langle \Omega, D_\alpha (\sum_{w \in [d]^*} \lambda_w e_w)\rangle_T=\langle\sum_{m=1}^\infty(-1)^m\xi_{m,\alpha}, \sum_{w\in[d]^*}\lambda_{w}e_{w}\rangle_T
\end{align*}
so long as the sum $\sum_{m=1}^\infty(-1)^m\xi_{m,\alpha}$ is a well-defined vector. We show that it indeed is and converges in norm:
\begin{align}\label{eq:estimate for xi_m,alpha}
\begin{split}
    \|\xi_{m,\alpha}\|_T&\leq \sum_{\beta\in [d]}\sum_{|w|=m-1}q^{\frac{m(m-1)}{2}}\|{e}_\alpha\|_U \;\|\bar{e}_\beta\|_U\;  \|{L}_{w\beta}^*\|\;\|\bar{e}_w\|_T\\
    &\leq d\; B\sum_{|w|=m-1}q^{\frac{m(m-1)}{2}}C^m\|\bar{e}_w\|_T\\
    &=d^{m}B^mC^m\sqrt{[m-1]_q!} \;q^{\frac{m(m-1)}{2}}
    \end{split}
\end{align}
where $d=|[d]|$, $B=\max_{\gamma\in [d]}\|\Bar{e}_\gamma\|_U$, $C=\max_{\gamma\in [d]}\|l(e_\gamma)\|<\infty$ (by Lemma \ref{lem:free annihilation operators are bounded}) and $q=\max_{i,j}|q(i,j)|<1$; and the last equality follows by noting from Lemma \ref{eq:formula for inner product of simple tensors} that 
  \begin{equation}\label{eq:estimate of T-norm of e_w}\|\Bar{{e}}_w\|_T^2\leq \sum_{\sigma\in S_{m-1}}q^{\inv(\sigma)}B^{2(m-1)}=[m-1]_q!B^{2(m-1)}\end{equation} where $\inv(\sigma)$ denotes the number of inversions in the permutation $\sigma$, so  that 
  \[\sum_{|w|=m-1}\|\Bar{e}_w\|_T\leq d^{m-1}B^{m-1}\sqrt{[m-1]_q!}.\]  Therefore we finally arrive at the claim:
\begin{align*}
    \sum_{m=1}^\infty\|\xi_{m,\alpha}\|\leq\sum_{m=1}^\infty q^{\frac{m(m-1)}{2}}d^m B^mC^m\sqrt{[m-1]_q!}<\;\infty
\end{align*}
by using ratio test. We have thus shown that
\begin{equation}\label{eq:expression for xi_alpha}
\xi_\alpha:=\sum_{m=1}^\infty(-1)^m\xi_{m,\alpha}\in \FockTH
\end{equation}
and $\langle\Omega, D_\alpha(\cdot)\rangle_T=\langle\xi_\alpha, \cdot\rangle_T$ on $\FockTHalg$; hence the functional extends to a bounded linear map on $\FockTH$. This proves that $\Omega\in \dom D_\alpha^*$ and $\xi_\alpha=D_\alpha^*(\Omega)$.
\end{proof}

In our next result, we  use  Bo\.zejko-Haagerup type inequality (see \cite[Lemma 2]{Nou1}): 
    For any $\xi\in \hu^{\otimes n}, n\geq 1$, we have
    \begin{align}\label{eq:Bozejko inequality}
        \|\xi\|_T\leq \|W(\xi)\|\leq C_q^{3/2}(n+1)\|\xi\|_T
    \end{align}
    where as before $q=\sup_{i,j}|q{(i,j)}|$ and $C_q=\prod_{n=1}^\infty(1-q^n)^{-1}$.
With the notations as in the proof of Theorem \ref{thm:existence of conjugate variables}, we establish the following:

\begin{cor}\label{cor:conjugates variables are in M Omega}
The conjugate system $(\xi_1,\ldots,\xi_d)$ corresponding to $(D_1, \ldots, D_{d})$ is given by $\xi_{\alpha}=D_{\alpha}^*\Omega=\sum_{m=1}^\infty(-1)^{m-1}\xi_{m,{\alpha}}$ for $1\leq \alpha\leq d:=\dim{\hr}$ where for $ m\geq 1$,
\begin{equation*}
    \xi_{m, \alpha}=\sum_{\beta\in [d]}\sum_{|w|=m-1}\prod_{{1\leq v< u\leq m-1}}q(i^{(w)}_{v}, i^{(w)}_u)\prod_{t=1}^{m-1}q(i_\alpha, i^{(w)}_t)\langle e_\alpha, \Bar{e}_{\beta}\rangle_U  {L}_{w\beta}^*(\bar{e}_w).
\end{equation*}
Moreover the sum $\sum_{m=1}^\infty(-1)^{m-1}W(\xi_{m, \alpha})$ converges in operator norm i.e. $\xi_{\alpha}\in \M_T\Omega$.
\end{cor}
\begin{proof}
    We have already seen the expression for the conjugate variables $\xi_{\alpha}$ in \eqref{eq:expression for xi_alpha} in Theorem \ref{thm:existence of conjugate variables}.
Note that the vector $L_{w\beta}^*(\bar{e}_w)$ belongs to the particle space $\hu^{\otimes{2m-1}}$ for any word $w$ of length $m-1$ and $\beta\in [d]$. Arguing as in \cite{MS},  we invoke  Bo\.zejko's Haagerup type inequality as mentioned in \eqref{eq:Bozejko inequality} to infer that
  \begin{align*}
      \|W(\xi_{m,\alpha})\|&\leq \sum_{\beta\in [d]}\sum_{|w|=m-1}q^{\frac{m(m-1)}{2}}\|{e}_{\alpha}\|_U\; \|\bar{e}_\beta\|_U\;  \|W({L}_{w\beta}^*\bar{e}_w)\|\\
      &\leq B\;  \sum_{\beta\in [d]}\sum_{|w|=m-1}q^{\frac{m(m-1)}{2}}\;C_q^{3/2}(2m) \;  \|{L}_{w\beta}^*(\bar{e}_w)\|_T\\
      &\leq  C_q^{3/2}d^{m} q^{\frac{m(m-1)}{2}} (2m) \;  C^{m} B^{m}\sqrt{[m-1]_q!}
  \end{align*}
  where $B,C$ are as written towards the end of proof of Theorem \ref{thm:existence of conjugate variables} and we have used   estimates as in \eqref{eq:estimate for xi_m,alpha} and  \eqref{eq:estimate of T-norm of e_w}. Thus, by using the ratio test, it follows that
  \begin{equation*}
      \sum_{m=1}^\infty\|W(\xi_{m,\alpha})\|\leq C_q^{3/2} \sum_{m=1}^\infty d^{m}\; q^{\frac{m(m-1)}{2}} (2m) \;  C^{m} B^{m}\sqrt{[m-1]_q!}<\infty
  \end{equation*}
  concluding the proof.
\end{proof}

\begin{rem} 
(1) In the expression for the conjugate variables given in \cite[Corollary 4.7]{MS}, certain products of free \emph{right} annihilation operators appear, while in our result products of \emph{left} annihilation operators appear. One can however verify that in the $q$-Gaussian case, the two situations indeed agree. This follows by noting that if ${R}_{\beta w}$ is  a product of appropriate right free annihilation operators then 
\begin{align*}\sum_{|v|=2m-1}\sum_{|w|=m-1}\lambda_v\langle  \bar{e}_w,  {R}_{\beta w }e_v\rangle_q
&=\sum_{|w'|=m-1}\sum_{|w|=m-1}\lambda_{w'\beta w}\langle {e}_{w}, e_{w'}\rangle_q\\
&=\sum_{|w|=m-1}\sum_{|w'|=m-1}\lambda_{w'\beta w}\langle e_{w'}, e_{w}\rangle_q\\
&=\sum_{|v|=2m-1}\sum_{|w'|=m-1}\lambda_v\langle\Bar{e}_{w'}, {L}_{ w'\beta}e_v\rangle_q
\end{align*}
simply because in the $q$-Gaussian case, one chooses an orthonormal basis $\{e_1,\cdots, e_d\}$ with $e_i\in \hr$ so that $\bar{e}_w=e_w$ and $\langle e_w, e_{w'}\rangle_T=\langle e_{w'}, e_w\rangle_T$ as they are always real for any two words $w,w'$ (summation of the above  form appears in the proof of Theorem \ref{thm:existence of conjugate variables}). In particular, we can replace $L_{{\beta w}}$ by $R_{\beta w}$ in Corollary \ref{cor:conjugates variables are in M Omega} which produces exactly the same formula as in \cite[Corollary 4.7]{MS}. Note that  equality of above forms no longer holds in non-tracial settings, so left free annihilation operators seem to be  more convenient choice.

(2) The orthonormal set $\{e_1,\cdots,e_d\}$ can be chosen to be eigenvectors of the analytic generator $A$ of $U_t$ so that   each element $W(e_k)$ is  analytic  in $\M_T$ (w.r.t the modular automorphism group $\sigma^\varphi$ of the canonical state $\varphi$); consequently the conjugate variables $\xi_i$'s are also analytic as they are norm-limits of vectors which are linear combinations of  tensors of $e_i$'s.

(3) When $q_{ij}=q$ for all $i,j$,  we get
\[ \xi_{ \alpha}=\sum_{m=1}^\infty (-1)^{m-1}\sum_{\beta\in [d]}\sum_{|w|=m-1}q^{\frac{m(m-1)}{2}}\langle e_\alpha, \Bar{e}_{\beta}\rangle_U  {L}_{w\beta}^*(\bar{e}_w),\]
which is slightly different that the formula in \cite[Remark 1.5]{KSW} because the initial assumption on the set $\{e_1,\ldots,e_d\}\subseteq\hc$ in \cite{KSW} is that it is orthonormal in the {undeformed} inner product of $\hc$ while our assumption is that the set is orthonormal in the deformed inner product of $\hu$. In particular,
if $q=0$ then $\xi_{m,\alpha}=0$ for all $m\geq 2$, so that \[\xi_\alpha=\xi_{1,\alpha}=\sum_{\beta\in [d]}\langle e_\alpha, \Bar{e}_\beta\rangle_U l(\beta)^*\Omega=\sum_{\beta\in [d]}\langle e_\alpha, \bar{e}_\beta\rangle_U\;e_\beta.\]
\end{rem}

\subsection*{Lipschitz Conjugate Variables}\label{sec:Lipschitz}
Our aim now is to prove that the conjugate variables $(\xi_1,\ldots,\xi_d)$ as obtained  above
are {\em Lipschitz} i.e. \[\xi_\ell\in\dom\bar{\partial}_k \;\;\text{and}\;\;  \bar{\partial}_k(\xi_\ell )\in \M_T\otimes\M_T^{\op}\] for all $1\leq k,\ell \leq d$, where $\partial_k$ is the quasi-free difference quotients as defined in the beginning of this Section. For this purpose, 
 we consider a different set of rules for partitions and counting their crossings exactly as in \cite[Section 5]{MS}. The partitions again consist of singletons and pairs, and we draw them 
in the following ways to locate crossings:
\begin{enumerate}
    \item Consider $n+1$ vertices $n>\cdots>1>0.$
    \item The vertex $0$ must be coupled with some $k\in \{1,\ldots, n\}$ with height 1.
    \item Each $\ell\in \{1,\ldots, k-1\}$ is a singleton or is coupled with one of $\{k+1,\ldots,n\}$ with height $\ell+1$
    \item Vertices which are not coupled with one of $\{1,\ldots, k-1\}$ should be singletons and are drawn with straight lines to the top.
\end{enumerate}
Let $C(n+1)$ be the set of partitions defined by the rules above. For each $\pi\in C(n+1)$, denote by $s_\ell (\pi)$ and $s_r(\pi)$ respectively the set of singletons in the left area $n\geq l>\pi(0)$ and the set of singletons in the right area $\pi(0)>\ell\geq 1$. Let $p(\pi)$ denote the set of pairings. 

We now give a formula for the action of $\partial_k$'s on the algebra $\C[A_1,\ldots,A_d]$. The algebra $\C[A_1,\ldots,A_d]$ is nothing but the space $\Span\{W(e_w); w\in [d]^*\}$. This  follows by  noting the following calculation (see eq. \eqref{eq:annihilation operator acting on simple tensors}): \begin{align}\label{eq:expression for product of Wick operators}W(\xi_{n+1}\cdots\xi_{1})=W(\xi_{n+1})W(\xi_n\cdots\xi_{1})-\sum_{k=1}^{n}\langle\bar{\xi}_{n+1},\xi_k\rangle_U\prod_{t=1}^{n-k}q(i_k, i_{k+t})W(\xi_n\cdots\widehat{\xi}_k\cdots\xi_1)\end{align}
whenever $\xi_k\in\hu^{(i_k)}$ for $i_k\in{N}$, $1\leq k\leq n+1$. Therefore it suffices to understand the action of $\partial_k$'s on the operators $W(e_w)$ for each word $w$.

\begin{prop}\label{prop:expression for partial}
    For each $\alpha\in[d]$ and $w=\alpha_{n}\cdots\alpha_1\in[d]^* $, $n\geq1$, we have
    \begin{equation*}
        \partial_\alpha W(e_w)=\sum_{\pi\in C(n+1)}(-1)^{|p(\pi)|-1} {D}(\pi)\; {E}(\pi) W(e_{s_\ell (\pi)})\otimes W(e_{s_r(\pi)})^{\op}
    \end{equation*}
    where     \begin{align*}
       & {D(\pi)}:=\prod_{\underset{0\leq v<u<\pi(0)}{u,v \text{ belong to a pairing}}}q(i_v, i_u)\cdot\prod_{\substack{u, v\text{ belong to a pairing}\\{0< v<u\leq \pi(0)}\\
       \pi(u)>\pi(v)} } q(i_{\pi(v)}, i_{\pi(u)})\\
       &\quad\quad\quad\quad\quad\quad\quad\quad\cdot\prod_{\underset{v\notin s_r(\pi), u\in s_r(\pi)}{0<v<u<\pi(0)}}q(i_v,i_u)
\cdot\prod_{\underset{v\in s_\ell(\pi), u\notin s_\ell(\pi)}{\pi(0)<v<u\leq n}} q(i_v,i_u)\\
       & {E}(\pi):= \prod_{\underset{{k>\ell}}{(k,\ell)}\in p(\pi)}\langle \bar{e}_{\alpha_k}, e_{\alpha_\ell}\rangle_U\\
&e_{s_\ell (\pi)}=e_{\alpha_{j_r}}\cdots e_{\alpha_{j_1}}\;\;\text{  for }\;\; {s_\ell (\pi)}=\{j_r>\ldots >j_1\} \;\;\;\text{ and }\\
&e_{s_r(\pi)}=e_{\alpha_{j'_s}}\cdots e_{\alpha_{j'_1}}\;\;\text{ for }\;\; s_r(\pi)=\{j'_s>\ldots >j'_1\},
    \end{align*}
    with the prescription that $\alpha_0=\alpha$.
\end{prop}
\begin{rem}
    As in Proposition \ref{prop: Algebraic formula for dual variable D_alpha}, we have divided the expression for $D(\pi, w)$ into four products. The terms in the first product (resp. the second product) correspond to the crossings occurring on the line of the vertex $`u$' (resp. $`\pi(u)$') of any pair $(u, \pi(u))$ for $0<u<\pi(0)$ with the lines of vertices  of another pair.
    The third product corresponds to the crossings in  the right side of $\pi(0)$ on the lines of singletons  with all pairs except $(0,\pi(0))$, while the fourth product corresponds to the crossings on the lines of  singletons in the left side of $\pi(0)$.
\end{rem}
\begin{proof}
    As in  \cite[Proposition 5.1]{MS}, the proof goes by induction on $n$. For $n=1$, the expression  is obvious because \[\partial_\alpha(W(e_{\alpha_1}))=\partial_\alpha(A_{\alpha_1})=\langle \bar{e}_{\alpha_1}, e_\alpha\rangle_U \mathds{1}\otimes\mathds{1}.\] So assume this to be true for some $n\geq 1$. Use then  \eqref{eq:expression for product of Wick operators} to calculate:
    \begin{align*}
        \partial_\alpha &W(e_{\alpha_{n+1}\cdots \alpha_1})=\partial_\alpha\left(W(e_{\alpha_{n+1}})W(e_{\alpha_n\cdots \alpha_1})-\sum_{u=1}^n  \prod_{u<t\leq n} q(i_{u},i_{t}) \langle \bar{e}_{\alpha_{n+1}} , e_{\alpha_u}\rangle_U\; W(e_{\alpha_n \ldots \widehat{\alpha}_u \ldots \alpha_1}) \right)\\
        &=\langle \Bar{e}_{\alpha_{n+1}}, e_{\alpha}\rangle_U \mathds{1}\otimes W(e_{\alpha_n\cdots \alpha_1})^{\op}+(W(e_{\alpha_{n+1}})\otimes \mathds{1})\partial_\alpha W(e_{\alpha_n\cdots\alpha_1})\\
        &\quad\quad\quad\quad\quad-\sum_{v=1}^n  \prod_{v<t\leq n} q(i_{v},i_{t}) \langle \bar{e}_{\alpha_{n+1}} , e_{\alpha_v}\rangle_U\;\partial_\alpha W( e_{\alpha_n \ldots \widehat{\alpha}_u \ldots \alpha_1}). 
    \end{align*}
By using induction argument, we separately compute the last two terms in the above sum:
 \begin{align}\label{eq:Aalphan+1}
   \begin{split}  &(W(e_{\alpha_{n+1}})\otimes \mathds{1})\partial_\alpha W(e_{\alpha_n\cdots\alpha_1})=\sum_{\pi\in C(n+1)}(-1)^{|p(\pi)|-1} {D}(\pi) {E}(\pi) W(e_{\alpha_{n+1}})W(e_{s_\ell (\pi)})\otimes W(e_{s_r(\pi)})^{\op}\\
     &=\sum_{\pi\in C(n+1)}(-1)^{|p(\pi)|-1} {D}(\pi) {E}(\pi) W(e_{\alpha_{n+1}}e_{s_\ell (\pi)}) \otimes W(e_{s_r(\pi)})^{\op}\\
     &+\sum_{\pi\in C(n+1)}(-1)^{|p(\pi)|-1} {D}(\pi) {E}(\pi) \sum_{u=1}^{n(\pi)}  \prod_{t=u+1}^{n(\pi)} q(i_{j_u^\pi},i_{j_t^\pi}) \langle \bar{e}_{\alpha_{n+1}} , e_{\alpha_{j_u^{\pi}}}\rangle_U W(e_{\alpha_{j_{n(\pi)}^\pi} \ldots \hat{\alpha}_{j_u^\pi} \ldots \alpha_{j_1^\pi}}) \otimes W(e_{s_r(\pi)})^{\op}
     \end{split}
 \end{align}
 where $n(\pi)=|s_\ell(\pi)|$ and $s_{\ell}(\pi)=\{j_{n(\pi)}^{\pi}>\ldots >j_1^{\pi}\}$;  then compute
 \begin{align}\label{eq:partialbetaonealphan...1}
 \begin{split}
     \sum_{v=1}^n & \prod_{v<t\leq n} q(i_{v},i_{t}) \langle \bar{e}_{\alpha_{n+1}} , e_{\alpha_v}\rangle_U\;\partial_\alpha W( e_{\alpha_n \ldots \widehat{\alpha_v} \ldots \alpha_1})\\
     &=\sum_{\sigma\in C(n)}\sum_{v=1}^n  \prod_{v<t\leq n} q(i_{v},i_{t}) \langle \bar{e}_{\alpha_{n+1}} , e_{\alpha_v}\rangle_U(-1)^{|p(\sigma)|-1}{D}(\sigma, v){E}(\sigma, v) W(e_{s_\ell (\sigma)})\otimes W(e_{s_r(\sigma)})^{\op}
     \end{split}
 \end{align}
 where $D(\sigma, v)$ and $E(\sigma, v)$ are written respectively  for $D(\sigma)$ and $E(\sigma)$ only but just to emphasize that   each $\sigma\in C(n)$ acts on the word $\alpha_n\cdots\widehat{\alpha}_v\cdots\alpha_1$. Note that each term of the last sum of \eqref{eq:Aalphan+1} appears  as   one of the terms in the last sum of \eqref{eq:partialbetaonealphan...1}. To see this, we follow the similar argument as in Theorem \ref{thm:existence of conjugate variables}. Let $\pi\in C(n+1)$ and $u\in \{1,\ldots, n(\pi)\}$, where $n(\pi)=|s_\ell (\pi)|$. 
 Set $v=s_\ell (\pi)_u=j_u^\pi$ and obtain the partition $\sigma\in C(n)$ of the set $\{0,\ldots, n\}\setminus\{v\}$ in an order-preserving fashion by removing the singleton $s_\ell(\pi)_u=v$. Note that $p(\pi)=p(\sigma), s_\ell (\pi)=s_\ell {(\sigma)}\setminus\{v\}$ and $s_r(\pi)=s_r(\sigma)$. Hence ${E}(\pi)={E}(\sigma)$. Also note that the extra crossings in $\pi$ as compared to $\sigma$ may be occurring only due to the straight line at $v=j_u^\pi$. Therefore we have
 \begin{align*}
     {{D}(\pi)}={{D}(\sigma,v)}\cdot\prod_{\underset{t\notin s_\ell (\pi)}{t>v}}q(i_v, i_t)=\frac{\prod_{v<t\leq n}q(i_v,i_t)}{\prod_{\underset{t\in s_\ell (\pi)}{v<t\leq n}}q(i_v,i_t)}=\frac{\prod_{v<t\leq n}q(i_v,i_t)}{\prod_{{u<r\leq n(\pi)}}q(i_{j_u^\pi},i_{j_r^\pi})}
     \end{align*}
   where it is to be understood that $\frac{0}{0}$ is $0$.  This shows our claim that  all terms of the last sum of \eqref{eq:Aalphan+1} appear  in the last sum of \eqref{eq:partialbetaonealphan...1}, and the same time note that the remaining terms in \eqref{eq:partialbetaonealphan...1} correspond to those $\sigma$ and $v$ for which $v\leq \sigma(0)$. Again as in Theorem \ref{thm:existence of conjugate variables}, we identify such pair $(\sigma, v)$ i.e.  $\sigma\in C(n)$ and $v\leq \sigma(0)$, with a partition $\sigma'\in C(n+2)$ which creates a new pair between $v$ and $n+1$ with the requirement that $\sigma'\setminus(n+1, v)=\sigma$.  Note that $\sigma'(0)=\sigma(0)+1$, $s_\ell (\sigma')=s_\ell (\sigma), s_r(\sigma')=s_r(\sigma)$ and $|p(\sigma')|=|p(\sigma)|+1$. Also note that the extra crossings in $\sigma'$ occur with all the pairs (twice) and singletons (once) to the left of $v$ or on the line of $v$; so that whenever $E(\sigma)\neq 0$, then we have
     \begin{align*}
         {D}(\sigma')={D}(\sigma)\cdot\prod_{v< t\leq n}q(i_v, i_t)
     \end{align*}
     which follows by  similar arguments as in \eqref{eq:expression for Q(pi')} from Proposition \ref{prop: Algebraic formula for dual variable D_alpha}.
     In particular, we get 
     \[E(\sigma'){D}(\sigma')=E(\sigma){D}(\sigma)\cdot\prod_{v< t\leq n}q(i_v, i_t).\]
     Finally by subtracting \eqref{eq:Aalphan+1} from \eqref{eq:partialbetaonealphan...1} yields  
     \begin{align*}
         \partial_\alpha W(e_{\alpha_{n+1}\cdots\alpha_1})&=\langle \Bar{e}_{\alpha_{n+1}}, e_{\alpha}\rangle_U \mathds{1}\otimes W(e_{\alpha_n\cdots \alpha_1})^{\op}\\
         &+\sum_{\pi\in C(n+1)}(-1)^{|p(\pi)|-1} {D}(\pi)\; {E}(\pi)\; W(e_{\alpha_{n+1}}e_{s_\ell (\pi)}) \otimes W(e_{s_r(\pi)})^{\op}\\
         &+\sum_{\substack{{\sigma'\in C(n+2)}\\ \sigma'(n+1)\text{ is not a singleton}\\
         \sigma'(0)\neq n+1}}  \langle \bar{e}_{\alpha_{n+1}} , e_{\alpha_v}\rangle_U(-1)^{|p(\sigma)|-1}{D}(\sigma'){E}(\sigma') W(e_{s_\ell (\sigma)})\otimes W(e_{s_r(\sigma)})^{\op}\\
         &=\sum_{{\sigma'\in C(n+2)}}  (-1)^{|p(\sigma)|-1}{D}(\sigma'){E}(\sigma')\; W(e_{s_\ell (\sigma)})\otimes W(e_{s_r(\sigma)})^{\op}
     \end{align*}
     where the last equality follows by noting that $\langle \bar{e}_{{\alpha_{n+1}}}, e_\alpha\rangle_U\mathds
     {1}\otimes W(e_{\alpha_n\cdots\alpha_1})^{\op}$ corresponds to the partition in $C(n+2)$ which connects  $0$ to $n+1$ (in which case all other points are right singletons); while each $\pi\in C(n+1)$ corresponds to those partitions $\sigma'\in C(n+2)$ such that $n+1$ is a singleton, in which case we have $p(\sigma')=p(\pi), {D}(\sigma')={D(\pi)}, s_r(\sigma')=s_r(\pi)$ and $ s_\ell (\sigma')=\{n+1\}\cup s_\ell (\pi)$ so that $e_{\alpha_{n+1}}e_{s(\pi)}=e_{s(\sigma')}$. This completes the proof by induction.
\end{proof}

\begin{rem}
    One can also obtain Theorem \ref{thm:existence of conjugate variables}  using the formulas in  Proposition \ref{prop:expression for partial}  for the  quasi-free difference quotients as the conjugate variables $\xi_\alpha$'s are nothing but $\partial_\alpha^*(\mathds{1}\otimes\mathds{1})$ and it is an easy check that the expression $\langle\mathds{1}\otimes\mathds{1}, \partial_\alpha(\sum_{w\in [d]^*}\lambda_wW(e_w))\rangle_{L^2(\M_T\otimes\M_T^{\op})}$ is exactly the one in eq. \eqref{eq:expression for the inner product of omegea, D_alpha omega}. In fact, one can obtain  Proposition \ref{prop: Algebraic formula for dual variable D_alpha}  by checking that $D_\alpha:=(\id\otimes\varphi^{\op})\partial_\alpha$ form a normalized dual system for $1\leq \alpha\leq d$.
\end{rem}

We keep the notations of Corollary \ref{cor:conjugates variables are in M Omega} and Proposition \ref{prop:expression for partial} for the following result.
\begin{cor}\label{cor: Lipstisz conjugate}
The system of conjugate variables $(\xi_1,\ldots,\xi_d)$ corresponding to the tuple $(D_1,\ldots, D_d)$ is Lipschitz.
\end{cor}
\begin{proof}
   Fix $\alpha,\beta\in [d]$. First note that each $\xi_{m,\alpha}\in \FockTHalg$ so that $ W(\xi_{m,\alpha})\in\dom\partial_\beta$; hence by Proposition \ref{prop:expression for partial}, $\partial_\beta W(\xi_{m,\alpha})\in \M_T\Bar{\otimes}\M_T^{\op}$. We  show that $\sum_{m=1}^\infty\|\partial_\beta W(\xi_{m,\alpha})\|<\infty$, from which it will follow that $\sum_{m=1}^\infty(-1)^{m-1}\partial_\beta W(\xi_{m,\alpha})$ converges both in norm and in $L^2(\M_T\Bar{\otimes}\M_T^{\op})$ proving that  $ W(\xi_\alpha)\in \dom\bar{\partial}_\beta$ and $\Bar{\partial}_\beta W(\xi_\alpha)\in \M_T\bar{\otimes}\M_T^{\op}$.   Assume that $e_\alpha\in \h_\C^{(i_\alpha)}$ so that
    \begin{align*}
        \xi_{m, \alpha}=\sum_{\gamma\in [d]}\sum_{|w|=m-1}\prod_{{1\leq v< u\leq m-1}}q(i^{(w)}_{v}, i^{(w)}_u)\prod_{t=1}^{m-1}q(i_\alpha, i^{(w)}_t)\langle e_\alpha, \Bar{e}_{\gamma}\rangle_U  {L}_{w\gamma}^*(\bar{e}_w).
    \end{align*}
By Proposition \ref{prop:expression for partial},    note that for each word $w'$ of length $n$
\begin{align*}
    \|\partial_\beta 
 W(e_{w'})\|\leq \sum_{\pi\in C(n+1)} |{D}(\pi)|\; |{E}(\pi)| \;\|W(e_{s_\ell (\pi)})\|\;\|W( e_{s_r(\pi)})\|.
\end{align*}
But for each $\pi\in C(n+1)$, we have
\begin{align*}
    &|{D}(\pi)|\leq 1,\\
    &|{E}(\pi)|\leq B^{|p(\pi)|}\leq B^{n} \;\;\;\mbox{ where }{B=\max_{\gamma\in[d]}\{1,\|\bar{e}_\gamma\|_U\}},\\
    &\|W(e_{s_\ell (\pi)})\|\leq C_q^{3/2}(|s_\ell (\pi)|+1)\|e_{s_\ell (\pi)}\|_T\leq C_q^{3/2}(|s_\ell (\pi)|+1)B^{|s_\ell(\pi)|}\sqrt{[|s_\ell (\pi)|]_q!}\\&\quad\quad\quad\quad\quad\;\leq C_q^{3/2}(n+1)B^n\sqrt{[n]_q!},\\
    &\|W(e_{s_r(\pi)})\|\leq C_q^{3/2}(n+1)B^n\sqrt{[n]_q!}
\end{align*}
    where the last two inequalities are due to Bo\.zejko-Haagerup type inequality and the  estimate as in \eqref{eq:estimate of T-norm of e_w}. Consequently we  obtain
    \begin{align}
    \begin{split}\label{eq:norm of partial W(ew)}
         \|\partial_\beta W(e_{w'})\|&\leq |C(n+1)|\;B^{3n}\;C_q^3\;(n+1)^2\;[n]_q!\\
         &\leq (n+1)!\;B^{3n}\;C_q^3\;(n+1)^2\;[n]_q!,
         \end{split}
    \end{align}
    where we have used the fact that the cardinality $|C(n+1)|$ of $C(n+1)$ is upper-bounded by $(n+1)!$ as we can think $C(n+1)$ as a subset of the symmetric group $S_{n+1}$.
    Note that the upper-bound of each $\|\partial_\beta 
 W(e_{w'})\|$ depends only on the length of the word $w'$ and is independent of the choice of the word itself. In particular,  since for each fixed word $w$ of length $m-1$ and $\gamma\in[d]$, ${L}^*_{w\gamma}(\Bar{e}_w)$ belongs to the subspace spanned by $\{e_{w'}
\}_{|w'|=2m-1}$, we write ${L}^*_{w\gamma}(\Bar{e}_w)=\sum_{|w'|=2m-1}\lambda_{w'}e_{w'}$ for some scalars $\lambda_{w'}\in\C$ and then estimate
    \begin{align*}
        \|\partial_\beta W({L}^*_{w\gamma}\Bar{e}_w)\|&\leq \sum_{|w'|=2m-1}|\lambda_{w'}|\;\|\partial_\beta 
 W(e_{w'})\|\\
        &\leq \sum_{|w'|=2m-1}|\lambda_{w'}|\;(2m)!\;B^{6m-3}\;C_q^3\;(2m)^2\;[2m-1]_q!.
    \end{align*}
    Further observe that for all $|w'|=2m-1$,  we have $\lambda_{w'}=L_{w'}L^*_{\gamma w}(\bar{e}_w)$ 
    so that \[|\lambda_{w'}|\leq \|L_{w'}\|\;\|{L}^*_{w\gamma}\|\;\|\Bar{e}_w\|_T\leq C^{3m-1}B^{{m-1}}\sqrt{[m-1]_q!}, \]
    where  as in Theorem \ref{thm:existence of conjugate variables}, $C=\max_{\gamma\in[d]}\|{l}(e_\gamma)\|$, $B$ is as above and the estimate of \eqref{eq:estimate of T-norm of e_w} is used. It follows that 
    \begin{align*}
        \|\partial_\beta W({L}^*_{w\gamma}\Bar{e}_w)\|\leq d^{2m-1}\;(2m)!\;C^{3m-1}\;B^{7m-4}\;C_q^3\;(2m)^2\;[2m-1]_q!\;\sqrt{[m-1]_q!}.
    \end{align*}
    Gathering all the estimates above, we finally arrive at the following:
    \begin{align*}
        \|\partial_\beta W(\xi_{m,\alpha})\|&\leq \sum_{\gamma\in[d]}\sum_{|w|=m-1}q^{\frac{m(m-1)}{2}}\|e_\alpha\|_U\;\|\Bar{e}_\gamma\|_U\;\|\partial_\beta W({L}^*_{w\gamma}\Bar{e}_w)\|\\
        &\leq q^{\frac{m(m-1)}{2}}\;d^{3m-1}\;(2m)!\;C^{3m-1}\;B^{7m-4}\;C_q^3\;(2m)^2\;[2m-1]_q!\;\sqrt{[m-1]_q!}
    \end{align*}
    therefore we conclude by using the ratio test that
    \begin{align*}
        \sum_{m=1}^\infty\|\partial_\beta W(\xi_{m,\alpha})\|\leq \sum_{m=1}^\infty q^{\frac{m(m-1)}{2}}\;d^{3m-1}(2m)!C^{3m-1}B^{7m-4}C_q^3\;(2m)^2[2m-1]_q!\;\sqrt{[m-1]_q!}<\infty
    \end{align*}
 thus   completing the proof.
\end{proof}

\subsection*{Power Series Expansion of Conjugate Variables}
The fact that the conjugate variables belong to $\M_T$ and are Lipschitz can also be obtained by writing each  $W(e_w)$ (for a word $w$) as a non-commutative polynomial in  $\{W(e_1), \ldots, W(e_d)\}$ and then by estimating their norms. This approach  avoids the use of Bo\.zejko-Haagerup inequality. The non-commutative power series thus obtained will have infinite radius of convergence.  We describe this method below by following the techniques of \cite{MS}. Also see \cite[Corllary 5.2]{Yang}.

For the purpose, we denote by $D(n)$ the collection of partitions of $\{1,\ldots,n\}$ consisting of either singleton or pairs for $n\geq 1$. As before, denote by $p(\pi)$ the set of pairs and by $s(\pi)$ the set of singletons for any partition   $\pi\in D(n)$. Also we set  $\pi(u)=v$ if $(u,v)\in p(\pi)$, and $\pi(u)=u$ if $u\in s(\pi)$.

\begin{prop}\label{prop:formula for W(e_w)}
    For $n\geq1$ and $\alpha_1,\ldots,\alpha_n\in [d]$, we have
    \begin{align*}
        W(e_{\alpha_n\cdots\alpha_1})=\sum_{\pi\in D(n)}(-1)^{|p(\pi)|}Q(\pi)E(\pi)W^{s(\pi)}
    \end{align*}
    where
    \begin{align*}
         &Q(\pi)=\prod_{\underset{v\in s(\pi)}{u<v<\pi(u)}}q(i_u,i_v)\cdot\prod_{u<v<\pi(v)<\pi(u)}q(i_u,i_v)q(i_{\pi(u)}, i_{\pi(v)})\cdot\prod_{u<v<\pi(u)<\pi(v)}q(i_u,i_v)\\
  &E(\pi):= \prod_{\underset{{k>\ell}}{(k,\ell)}\in p(\pi)}\langle \bar{e}_{\alpha_k}, e_{\alpha_\ell}\rangle_U\\
&W^{s(\pi)}:=W(e_{\alpha_{\ell_s}})\cdots W(e_{\alpha_{\ell_1}})\;\mbox{ for } s(\pi)=\{\ell_s>\ldots>\ell_1\}.
    \end{align*}
\end{prop}
\begin{proof}
Remark that the terms in $Q(\pi)$ correspond to the crossings on the line of $v$ (and on $\pi(v)$ in the second product) as $v$ varies. Similar to  Remark \ref{rem: different formulas for Q(pi)}, the terms $q(i_u,i_v)$ in the product $Q(\pi)E(\pi)$  can be replaced by $q(i_{\pi(u)}, i_{\pi(v)}), q(i_{u}, i_{\pi(v)})$ etc because $E(\pi)$ is non-zero only if $i_u=i_{\pi(u)}$ for any pair $(u,\pi(u))$.

We sketch the proof below which  goes by the induction on $n$. For $n=1$, the formula is obvious. So assume  the expression to be true for some $n\geq 1$. Then calculate  by using eq. \eqref{eq:expression for product of Wick operators} and by the induction argument:
\begin{align*}
    W(e_{\alpha_{n+1}\cdots {\alpha_1}})&=W(e_{\alpha_{n+1}})W(e_{\alpha_n\cdots {\alpha_1}})-\sum_{u=1}^n\prod_{u<t\leq n}q(i_u,i_t)\langle\bar{e}_{\alpha_{n+1}},e_{\alpha_u}\rangle_UW(e_{\alpha_n\cdots\hat{\alpha}_u\cdots\alpha_1})\\
    &=\sum_{\pi\in D(n)}(-1)^{|p(\pi)|}Q(\pi)E(\pi)W(e_{\alpha_{n+1}})W^{s(\pi)}\\
    &-\sum_{u=1}^n\prod_{u<t\leq n}q(i_u,i_t)\langle\bar{e}_{\alpha_{n+1}},e_{\alpha_u}\rangle_U\sum_{\sigma\in D(n-1)}(-1)^{|p(\sigma)|}Q(\sigma,u)E(\sigma,u)W^{s(\sigma,u )}
\end{align*}
where $Q(\sigma,u), E(\sigma,u)$ and $W^{s(\sigma,u)}$ denote $Q(\sigma), E(\sigma)$ and $W^{s(\sigma)}$ respectively with emphasis that $\sigma\in D(n-1)$ acts on the word $\alpha_{n}\cdots\hat{\alpha}_u\cdots\alpha_1$. The two sums above  give nothing but the expression  
$\sum_{\tilde{\pi}\in D(n+1)}(-1)^{|p(\tilde{\pi})|}Q(\Tilde{\pi})E(\tilde{\pi})W^{s(\Tilde{\pi})}$. To see this, note that the first sum corresponds to those $\Tilde{\pi}\in D(n+1)$ where $n+1$ is a singleton. Whereas for the terms in the second sum, for each $u\in \{1,\ldots,n\}$ and $\sigma\in D(n-1)$ acting on the word $\alpha_{n}\cdots\hat{\alpha}_u\cdots\alpha_1$, one constructs $\Tilde{\pi}_{\sigma,u}\in D(n+1)$ by connecting $u$ to $n+1$ so that $\Tilde{\pi}_{\sigma,u}\setminus(n+1,u)=\sigma$. Then  observe that $|p(\Tilde{\pi}_{\sigma,u})|=1+|p(\sigma)|$ and $s(\Tilde{\pi}_{\sigma,u})=s(\sigma,u)$ so that $(-1)^{|p(\Tilde{\pi}_{\sigma,u})|}=-(-1)^{|p(\sigma)|}$,  $W^{s(\Tilde{\pi}_{\sigma,u})}=W^{s(\sigma,u)}$ and $E(\Tilde{\pi}_{\sigma,u})=\langle\bar{e}_{\alpha_{n+1}},e_{\alpha_u}\rangle E(\sigma,u)$. Moreover, the extra crossings in $\Tilde{\pi}_{\sigma,u}$ as compared to $\sigma$ occur due to the pair $(n+1,u)$ with all the lines to its left; hence  we have $Q(\Tilde{\pi}_{\sigma,u})=Q(\sigma,u)\prod_{u<t\leq n}q(i_u,i_t)$. These observations complete the proof.
\end{proof}

\begin{rem} (1). 
    For any $n\geq1$ and $\alpha_1,\ldots,\alpha_n\in [d]$, we observe that \begin{equation}\label{eq:estimate for W(e_w)}
        |W(e_{\alpha_n\cdots\alpha_1})\|\leq  n! B^n A^n 
    \end{equation}  where $B=\max_{\gamma\in [d]}\{1, \|\bar{e}_\gamma\|_U\}$ and $A=\max_{\gamma\in [d]}\|W(e_\gamma)\|\leq 2C_q^{3/2}$. This  follows by noting that $|D(n)|\leq n!$, and for any $\pi\in D(n)$, we have $|Q(\pi)|\leq 1$, $|E(\pi)|\leq B^{|p(\pi)|}\leq  B^n$ and $\|W^{s(\pi)}\|\leq A^{|s(\pi)|}\leq  A^n$.

    (2). By proposition \ref{prop:formula for W(e_w)}, we get a concrete formula for each conjugate variable $\xi_{\alpha}$ as a non-commutative power series in  $W(e_{1}), \ldots, W(e_d)$. Moreover,  we can  show that the power series has infinite radius of convergence and reprove Corollary \ref{cor:conjugates variables are in M Omega} as follows: Indeed, with the notations as in Corollary \ref{cor:conjugates variables are in M Omega}, 
for any $m\geq 1$ and $\alpha\in [d]$, we have 
\begin{align*}
    \|W(\xi_{m,\alpha})\|\leq \sum_{\beta\in [d]}\sum_{|w|=m-1}q^{\frac{m(m-1)}{2}} B\|W(L_{w\beta}^*\Bar{e}_w)\|.
\end{align*}
As in the proof of  Corollary \ref{cor: Lipstisz conjugate}, write $L_{w\beta}^*\Bar{e}_w=\sum_{|w'|=2m-1}\lambda_{w'}e_{w'}$ where the scalars $\lambda_{w'}$ satisfy
\[|\lambda_{w'}|\leq C^{3m-1}B^{m-1}\sqrt{[m-1]_q!}\] for $C=\max_{\gamma\in [d]}\|l(e_\gamma)\|$. By eq. \eqref{eq:estimate for W(e_w)} we get
\[\|W(L_{w\beta}^*\Bar{e}_w)\|\leq \sum_{|w'|=2m-1}|\lambda_{w'}|\;\|W(e_w')\|\leq d^{2m-1}C^{3m-1}B^{m-1}\sqrt{[m-1]_q!}(2m-1)!B^{2m-1}A^{2m-1}\] which further yields
\begin{align*}
    \|W(\xi_{m,\alpha})\|\leq d^{3m-1}q^{\frac{m(m-1)}{2}}C^{3m-1}B^{3m-1}\sqrt{[m-1]_q!}\;(2m-1)!\;A^{2m-1}
\end{align*}
and then ratio test says that $\sum_{m=1}^\infty\|W(\xi_{m,\alpha})\|<\infty$. By a similar argument, one can show that $\partial_\beta(W(\xi_\alpha))\in \M_T$ for all $\beta,\alpha\in [d]$ and hence $\xi_\alpha$ is Lipschitz.
\end{rem}

In terminologies of \cite{Brent}, the above results can be rephrased as follows. Let $\varphi=\langle\Omega,\cdot\Omega\rangle_T$ be the canonical normal faithful state on $\M_T$ and  let $\sigma^\varphi$ be the corresponding modular automorphism group. An element $x\in\M_T$ is  an {\em eigenoperator} (with respect to $\sigma^\varphi$) if $\sigma^\varphi_t(x)=\lambda^{it}x$ $\forall t\in\mathbb{R}$ for some $\lambda>0$. We choose the orthonormal set $\{e_1,\ldots,e_d\}$ in $\hu$ at the beginning of this Section in such a way that each $e_k\in\hc$ is an eigenvector for the analytic generator $A$ of $U_t$, so that $W(e_k)$ will be an eigenoperator for $\sigma^\varphi$ because $\sigma_t(W(e_k))=W(A^{-it}e_k)$ $ \forall t\in\mathbb{R}$, and that the set $\{W(e_1),\ldots, W(e_k)\}$ is self-adjoint (see \cite[Proof of Theorem 2.2]{Hiai} for such construction). Since the tuple $(W(e_1), \ldots, W(e_d))$ admits a system $(\xi_1,\ldots,\xi_d)$ of conjugate variables as shown in Theorem \ref{thm:existence of conjugate variables}, it has {\em finite free Fisher information} i.e. 
\begin{align*}
    \sum_{k=1}^d\|\xi_k\|_{L^2(\M_T)}^2<\infty,
\end{align*}
in the terminology of \cite[Definition 3.6]{Brent}. We summarise our discussions  in the following result:
\begin{cor}[see \cite{KSW}, Corollary 1.5]\label{cor:generators have finite free fisher information}
Let $(\h_{\mathbb{R}}=\oplus_{i\in{N}}\hr^{(i)}, U_t, (q_{ij})_{i,j\in{N}})$ be given  with $\hr$ finite dimensional (and hence  ${N}$ a finite set), the von Neumann algebra $\M_T$ equipped with the canonical state $\varphi$ is generated by a finite set $G=G^*$ of eigenoperators of the modular group of automorphisms $\sigma_t^\varphi$ with finite free Fisher information.
\end{cor}

\section{Factoriality and non-injectivity} \label{sec:main}
In this section, we prove that the mixed $q$-Araki-Woods von Neumann algebra $\M_T$ is a factor i.e. its center $\M_T\cap\M_T'=\C$, and that $\M_T$ is non-injective. Recall that a von Neumann algebra $M\subseteq\mathbf{B}(H)$ is \emph{injective} if there exists an idempotent from $\mathbf{B}(H)$ onto $M$ of norm $1$. Moreover we also discuss the type classifications of the algebras $\M_T$. Connes \cite{Con} defined the {\em $S$-invariant} of a factor $M$ with a separable predual  to be the intersection of the spectra of the modular operators $\Delta_\phi$ of all normal faithful states $\phi$. The factor $M$ is of type $\mathrm{III}$ if and only if $0\in S(M)$. In this case, the type classification of $M$ is given as follows:
\begin{align*}
    S(M)=\left\{\begin{array}{ll}
    \lambda^{\mathbb{Z}}\cup\{0\} & \mbox{if } M \mbox{ is of type } \mathrm{III}_\lambda, 0<\lambda<1,\\
    \left[0,\infty\right)& \mbox{if } M \mbox{ is of type }\mathrm{III}_1,\\
         \{0,1\}& \mbox{if } M \mbox{ is of type } \mathrm{III}_0.
    \end{array}\right.
\end{align*}
Recall that the centralizer of $M$ with respect to a normal faithful state $\phi$ is defined by $M^\phi=\{x\in M;  \phi(xy)=\phi(yx)\;\forall y\in M\}$. If the centralizer $M^\phi$ is a factor, then $S(M)$ is just the spectrum of  the modular operator $\Delta_{\phi}$ corresponding to $\phi$.
We now state the following abstract results of Nelson \cite{Brent}:

\begin{tw}[\cite{Brent}, Theorem A, Theorem B]\label{Thm:Nelson}
Let $M$ be a von Neumann algebra with a faithful normal state $\phi$. Suppose $M$ is generated by a finite set $G=G^{\ast}$, $|G|\geqslant 2$ of eigenoperators of the modular group $\sigma^{\phi}$ with finite free Fisher information. Then $(M^{\phi})^{\prime} \cap M = \mathbb{C}$. In particular, $M^{\phi}$ is a $\mathrm{II}_1$ factor and if $H < \mathbb{R}^{\times}_{\ast}$ is the closed subgroup generated by the eigenvalues of $G$ then $M$ is a factor of type
\[
\left\{\begin{array}{ll} \mathrm{III}_1 & \text{ if } H=\mathbb{R}^{\times}_{\ast} \\ \mathrm{III}_{\lambda} & \text{ if } H= \lambda^{\mathbb{Z}}, 0<\lambda<1 \\ \mathrm{II}_1 & \text{ if } H=\{1\}.  \end{array}   \right.
\]
Furthermore $M^\varphi$ is full, and if $M$ is a type $\mathrm{III}_{\lambda}$ factor, $0<\lambda<1$, then $M$ is full.
\end{tw}

Recall that a factor $M$ is {\em full} if for any bounded sequence $x_{n}$ in $M$ with $\|\phi(\cdot x_n)-\phi(x_n\cdot)\|\to 0$ for all normal states $\phi$ on $M$, there is a bounded sequence $\lambda_n\in\C$ such that $x_n-\lambda_n\to0$ $*$- strongly. Remark that a full factor is necessarily non-injective.

We mention in the passing that if $U_t$ is a strongly continuous orthogonal representations on a real Hilbert space $\hr$, then $\hr$ decomposes into a direct sum $\kr\oplus\lr$  of invariant subspaces of $U_t$, where $\mathsf{K}_U$ is  densely spanned by the eigenvectors of $U_t$. We call the part $(\kr, {U_t}_{|_{\kr}})$ as the {\em almost periodic} and the part $(\lr, {U_t}_{|_{\lr}})$ as {\em weakly mixing}. Note that if $\lr$ is non-zero, then it must be infinite-dimensional. Also the almost periodic part has a  nice decomposition as representations on  a direct sum of either one or two dimensional real spaces (see \cite{Sh1}). In particular if $\hr=\oplus_{i\in{N}}\hr^{(i)}$ is finite dimensional then each $\hr^{(i)}$ is almost periodic, hence decomposes as a direct sum of one or two dimensional invariant subspaces.
We are now ready to give our main results.  As before, we denote below  the canonical normal faithful state $\langle\Omega, \cdot\Omega\rangle_T$ on $\M$ by $\varphi$.

\begin{tw} \label{thm:factor}
Let $(\h_{\mathbb{R}}=\oplus_{i\in{N}}\hr^{(i)}, U_t, (q_{ij})_{i,j\in{N}})$ be given for some countable set $N$, with $\dim(\h_{\br})\geqslant 2$. Then $\M:=\Gamma_T(\hr, U_t)$ is a factor. Moreover, if $G < \mathbb{R}^{\times}_{\ast}$ is the closed subgroup generated by the spectrum of $A$ then $\mlg$ is a factor of type
\[
\left\{\begin{array}{ll} \mathrm{III}_1 & \text{ if } G=\mathbb{R}^{\times}_{\ast} \\ \mathrm{III}_{\lambda} & \text{ if } G= \lambda^{\mathbb{Z}}, 0<\lambda<1 \\ \mathrm{II}_1 & \text{ if } G=\{1\}.  \end{array}   \right.
\]
\end{tw}
\begin{proof}
If $\dim\hr<\infty$ then the factoriality and type classification of $\M$ follow from Corollary \ref{cor:generators have finite free fisher information} and Theorem \ref{Thm:Nelson}. Therefore assume that $\dim\hr=\infty$.
    If the weakly mixing part of $U_t$ is non-trivial, then $\M$ is  a type III$_1$ factor 
 by \cite[Theorem 5.9]{BKM}.
 
 So we may assume that  $U_t$ is almost periodic  on an infinite dimensional space $\hr$. We use an inductive limit argument to prove factoriality and type classification; this may be apparent to experts but we provide details for completeness. As mentioned above, one can find an increasing sequence $\hr (n)$ of $U_t$-invariant finite-dimensional subspaces of $\hr$ such that each $\hr(n)$ is compatible with the direct sum i.e. $\hr(n)=\oplus_{i\in{N}}(\hr(n)\cap \hr^{(i)})$, and that their union spans a dense subspace of $\hr$. Denote the von Neumann algebra $\Gamma_T(\hr(n), {U_t}_{|_{\hr(n)}})$ by $\M_n$ and the restriction of the state $\varphi$ to $\M_n$ by $\varphi_n$.  Through our assumption, we have  implication that the unital $\ast$-algebra $\cup_{n\geq 1}\M_n$  is strongly dense in $\M$. An appeal to \cite[Proposition 6.1]{BKM} entail that each $\M_n$ is naturally embedded in $\M$ as a subalgebra and that there is a $\varphi$-preserving conditional expectation $E_n$ from $\M$ onto $\M_n$, so that $\M_n$ is invariant under the modular automorphism group $\sigma_t^\varphi$ of the  state $\varphi$. Observe that $E_n(x)\to x$ strongly for all $x\in\M$. One  easily verifies that ${\M_n^{\varphi_n}}\subseteq \M^\varphi$.
 
 We now show that $\M^\varphi$ is a factor. Indeed 
 if $x\in ({\M^\varphi})'\cap \M$, then $E_n(x)\in ({\M_n^{\varphi_{n}}})'\cap \M_n$ for all $n\geq 1$; but we know from Corollary \ref{cor:generators have finite free fisher information} and Theorem \ref{Thm:Nelson}  that
$({\M_n^{\varphi_{n}}})'\cap \M_n=\C$ (since $\dim\hr(n)<\infty$), therefore $E_n(x)=\varphi_n(E_n(x))=\varphi(x)$ for all $n\geq 1$ which finally implies that $x=\varphi(x)$. We have thus shown that $({\M^\varphi})'\cap\M=\C$ so that $\M^\varphi$ is a factor; hence  $\M$ is a factor. Since $\M^\varphi$ is a factor, the  type classification of $\M$ follows from the spectral data of the modular operator $\Delta_\varphi$ of the state $\varphi$, which is related to the analytic generator $A$ as given  in \cite[Theorem 3.13]{BKM}.
\end{proof}

We now show that the factor $\M_T$ is non-injective with the same data as in Theorem \ref{thm:factor}.

\begin{tw} \label{thm:noninj}
The factor $\M=
\Gamma_T(\hr, U_t)$ is not injective as soon as $\dim(H_{\br}) \geqslant 2$.
\end{tw}
\begin{proof}
If $U_t$ has non-trivial weakly mixing part, then $\M$ is non-injective by \cite[Theorem 8.4]{BKM}. So we may assume that $U_t$ is almost periodic.
As in the proof of \cite[Theorem 3.5]{KSW},  we will find a  non-injective subalgebra with a conditional  expectation from the ambient factor $\M$. So let $\mathsf{K}_\R$ be a two-dimensional invariant subspace of $U_t$ compatible with the direct sum decomposition i.e. $\mathsf{K}_\R=\oplus_{i\in {N}}(\mathsf{K}_\R\cap \hr^{(i)})$. As in the proof of Theorem \ref{thm:factor} above, there is a conditional expectation from $\M$ to the subalgebra $\M_1:=\Gamma_T(\mathsf{K}_\R,{U_t}_{|_{\mathsf{K}_\R}})$.  Now if   ${U_t}_{|_{\mathsf{K}_\R}}$ is trivial then $\M_1=\M_1^{\varphi_{|_{\M_1}}}$ which in conjugation with Corollary \ref{cor:generators have finite free fisher information} and Theorem \ref{Thm:Nelson} yields that $\M_1$ is full and hence $\M_1$ is non-injective. On the other hand if ${U_t}_{|_{\mathsf{K}_\R}}$ is ergodic then by Theorem \ref{thm:factor}, $\M_1$ is a type III$_\lambda$ factor for some $0<\lambda<1$; the same combination of Corollary \ref{cor:generators have finite free fisher information} and Theorem \ref{Thm:Nelson} then concludes that $\M_1$ is full and hence non-injective. The proof is now complete.
\end{proof}

\smallskip
\noindent {\bf Acknowledgments.} 
The author is grateful to  Adam Skalski and Mateusz Wasilewski for carefully reading the initial draft and giving  valuable comments. He also thanks the anonymous referee for carefully reading the manuscript and providing valuable suggestions.  The author is  partially supported by the National Science Center (NCN) grant no. 2020/39/I/ST1/01566.

\end{document}